\documentclass[12pt]{amsart}
\usepackage{geometry} 
\usepackage{amsmath,amssymb,amsthm,verbatim}
\usepackage{hyperref}
\usepackage{color}
\usepackage{stmaryrd}
\usepackage{bm}

\allowdisplaybreaks

\newtheorem{theorem}[subsection]{Theorem}
\newtheorem{lemma}[subsection]{Lemma}
\newtheorem{cor}[subsection]{Corollary}

\theoremstyle{definition}

\newtheorem{remark}[subsection]{Remark}

\newtheorem{question}[section]{Question}

\newcommand{\dist}{\mathrm{dist}}

\newcommand{\haus}{\mathcal{H}}

\newcommand{\spt}{\mathrm{spt}}

\newcommand{\reg}{\mathrm{reg}}

\newcommand{\graph}{\mathrm{graph}}

\newcommand{\eps}{\epsilon}

\newcommand{\sing}{\mathrm{sing}}

\newcommand{\R}{\mathbb{R}}
\newcommand{\bC}{\mathbf{C}}
\newcommand{\C}{\mathbb{C}}
\newcommand{\N}{\mathbb{N}}
\newcommand{\Z}{\mathbb{Z}}
\newcommand{\bM}{\mathbf{M}}
\newcommand{\del}{\partial}

\newcommand{\CP}{\mathbb{CP}}

\newcommand{\cH}{\mathcal{H}}

\newcommand{\cA}{\mathcal{A}}

\newcommand{\cL}{\mathcal{L}}

\newcommand{\cE}{\mathcal{E}}

\newcommand{\osc}{\mathrm{osc}}
\newcommand{\cG}{\mathcal{G}}

\newcommand{\res}{\mathbin{\hspace{0.1em}\vrule height 1.3ex depth 0pt width 0.13ex\vrule height 0.13ex depth 0pt width 1.0ex}} 
\newcommand{\vertiii}[1]{{\left\vert\kern-0.25ex\left\vert\kern-0.25ex\left\vert #1 \right\vert\kern-0.25ex\right\vert\kern-0.25ex\right\vert}} 

\makeatletter 														
	\newcommand*{\ccone}{\mathrel{\mathpalette\@ccone\relax}} 
	\newcommand*{\@ccone}[2]{
	   \setbox\z@=\hbox{\m@th$#1\times$}
 	  \ooalign{\raise.25\ht\z@\copy\z@\cr\box\z@}
	}
	\makeatother

\newcommand{\<}{\langle}
\renewcommand{\>}{\rangle}

\newcommand{\Lip}{\textnormal{Lip}}

\usepackage{fancyhdr}
\begin{document}

\title{Entire area-minimizing surfaces in $\R^n$ of density two are quadratic or planar}

\author{Nick Edelen}
\address{Department of Mathematics, University of Notre Dame, Notre Dame, IN 46556 USA}
\email{nedelen@nd.edu}
\author{Luis Atzin Franco Reyna}
\address{Department of Mathematics, University of Notre Dame, Notre Dame, IN 46556 USA}
\email{lfrancor@nd.edu}
\author{Paul Minter}
\address{Department of Pure Mathematics and Mathematical Statistics, University of Cambridge}
\email{pdtwm2@cam.ac.uk}

\begin{abstract}
We classify any entire area-minimizing surface $M^2\subset\R^n$ with density $2$ at infinity as either planar, or the zero set of a quadratic holomorphic polynomial inside some affine copy of $\C^2$.  We also prove algebraicity of $M^2$ when the tangent cone at infinity has multiplicity one with general density.  In our proof we introduce a sheeting-type theorem for minimizing surfaces near a multiplicity-two plane at infinity, and a new proof of holomorphicity (different from \cite{micallef}) based on a general ``calibrated at infinity'' principle.
\end{abstract}

\maketitle

Motivated by the interactions between area-minimizing surfaces and holomorphic subvarieties, we are interested in the following question:
\begin{question}\label{question-start}
    What is the structure of an entire, area-minimizing surface in $\R^n$?
\end{question}
It is well-known that any holomorphic subvariety of $\C^n$ is area-minimizing in its homology class. To what extent the converse holds, namely whether an entire area-minimizing surface in $\R^n$ is necessarily holomorphic (up to rigid motion) has been asked by numerous authors \cite{micallef, schoen:codim, white:parkcity, yau:problems}.

Globally-defined minimal surfaces often exhibit strong rigidity properties, and perhaps the most famous example of this is the Bernstein problem for entire codimension-one minimal graphs (or, more generally, for area-minimizing or stable hypersurfaces) which asserts that such hypersurfaces must be flat in low dimensions  \cite{Fle62, DG65, Alm66, Sim68, dCP:R3, FCS:stable, P, SSY, B, CL:R4, CLMS24, Mazet}. Aside from being a natural classification question, understanding globally-defined minimal surfaces is in some sense ``dual'' to understanding the local degeneration of solutions, as one expects the small-scale picture of a surface near the formation of a singularity to be modeled on a global solution.

A starting point for Question \ref{question-start} is the work of Morgan \cite{Morgan_1982}, which classified dilation-invariant (``conical'') area-minimizers $T^2\subset\R^n$ as unions of holomorphic planes (holomorphic relative to some fixed orthogonal complex structure). In the special case $n=4$, i.e.~$M^2\subset\R^4\equiv \C^2$, Micallef \cite{micallef, micallef:note} gave a remarkable answer to Question \ref{question-start} by showing that in fact if $M$ is a complete, connected, oriented, parabolic, stable (branched) minimal surface, then it is holomorphic (up to a real rigid motion). In \cite{Edelen_Franco_Minter} we made explicit the further observation that if such $M^2\subset\R^4$ additionally has quadratic area growth, then in fact $M$ must be algebraic.

In higher codimension ($n \geq 5$), most progress has only been made when the surface is known a priori to have genus $0$ or $1$, primarily because of nice splitting properties of the complexified normal and tangent bundles.  Micallef \cite{micallef} showed that if $M^2\subset\R^n$ is complete, minimal, stable, genus $0$, and has finite total curvature, then $M$ must again be holomorphic up to a rigid motion, and is secretly contained within some copy of $\C^k\equiv \R^{2k}$.  Fraser--Schoen \cite{fraser-schoen} proved a holomorphicity result analogous to Micallef for genus $1$ surfaces which are ``covering stable,''  and recently Sagman--Thalmaier \cite{SaTh} extended the results of both Micallef and Fraser--Schoen by removing the assumptions of completeness and finite total curvature.  To illustrate the subtlety in general codimension, even in the case where the surface is smooth, there are examples of genus $2$ stable minimal immersions into $\R^{22}$ with finite total curvature which are not holomorphic with respect to any orthogonal complex structure (see \cite{ArMiPi, ArMi}).

\vspace{3mm}

For the general situation of area-minimizing surfaces as in Question \ref{question-start} (precisely, ``mass-minimizing integral $2$-currents $T$ without boundary''), it is known by the work of Rivière \cite{Riviere_2004} (see Theorem \ref{thm:unique}) that if $T$ has quadratic area growth, then $T$ must be asymptotic to a union of planes, possibly with multiplicity (and these planes must be jointly holomorphic for some orthogonal complex structure by \cite{Morgan_1982}). The total number of planes, counted with multiplicity, is equal to the density $\Theta_T(\infty) \in \N$ of $T$ at infinity. The quantity $\Theta_T(\infty)$ may be viewed as a measure of the complexity of $T$. The case $\Theta_T(\infty)=1$ is trivial, as then $T$ is necessarily a multiplicity-one plane (seen by, e.g., the monotonicity formula). For density $\Theta_T(\infty)=2$, there are non-trivial examples given by the zero set of any quadratic holomorphic polynomial in $\C^2$, e.g.~
\[
\{(z,w)\in \C^2: z^2=w\} \quad \text{or}\quad \{(z,w)\in \C^2:zw=1\}.
\]

Thus, area-minimizing surfaces $T\subset\R^n$ with $\Theta_T(\infty)=2$ are in some sense the ``simplest'' non-trivial area-minimizing surfaces.  Even so, their study (in particular in regards Question \ref{question-start}), is subtle, as a priori such surfaces could be very structurally complicated.  A key difficulty which illustrates the non-trivial subtleties of the problem is seen in the instance where $T$ is asymptotic to a multiplicity $2$ plane, as one could a priori imagine a situation where $T$ looks like two parallel planes connected by an infinite sequence of ``necks'' or ``handles'' which drift off to infinity. This problem of how to deal with ``infinite topology'' is a common difficulty in the study of minimal surfaces.

Our first main result is a resolution of Question \ref{question-start} when $\Theta_T(\infty)=2$. We completely classify such $T$, showing that $T$ must either be planar or must coincide with the zero set of a degree $2$ holomorphic polynomial in some copy of $\C^2$.

\begin{theorem}\label{thm:main}
    Let $T$ be an area-minimizing integral $2$-current in $\R^{2+m}$ (for $m\geq 2$) satisfying $\del T = 0$ and $\Theta_T(\infty)=2$. Identify $\C^2\equiv \R^4\times\{0^{m-2}\}\subseteq\R^{2+m}$. Then $T$ takes one of the following forms:
    \begin{enumerate}
        \item $T = 2\llbracket P\rrbracket$ for some affine $2$-plane $P\subset\R^{2+m}$;
        \item $T = A_\#\llbracket p_1+P_1\rrbracket + A_\#\llbracket p_2+P_2\rrbracket$ for $P_1,P_2$ holomorphic $2$-planes in $\C^2$, $p_1,p_2\in \R^{2+m}$, and $A\in O(2+m)$;
        \item $T = A_\#\llbracket p + \{(z,w)\in\C^2: az^2+bzw+cw^2+dz+ew+f = 0\}\rrbracket$ for some $a,b,c,d,e,f\in \C$, $p\in \R^{2+m}$, and $A\in O(2+m)$.
    \end{enumerate}
    In particular, $T$ is either planar, or $T$ is smooth, multiplicity one, genus zero, and is (up to a rigid motion) the zero set of a quadratic holomorphic polynomial.
\end{theorem}

We emphasize that Theorem \ref{thm:main} \emph{proves} that $T$ has genus $0$ using the weak asymptotics at infinity.

We also highlight the consequence of Theorem \ref{thm:main} that $T$ (if not planar) is necessarily contained in an affine copy of $\R^4$, which can be seen as a Liouville-type theorem controlling the codimension of $\spt\, T$ in terms of the density.  Indeed, for a general irreducible, minimizing integral $2$-current $T$ in $\R^{2+m}$ with $\del T = 0$ and which is sufficiently ``regular at infinity,'' one can prove that $\spt\, T$ is contained in an affine subspace of dimension $\leq 2\Theta_T(\infty)$.  It's worth noting that this bound holds (and is sharp) if $T$ is already known to be an irreducible algebraic curve in $\C^n$, as in this case each independent projective complex coordinate contributes one complex dimension to the hyperplane bundle, which itself has complex dimension $\leq \Theta_T(\infty) + 1$.  We also compare to the codimension bound of \cite[Corollary 0.6]{CM-complexity}, which implies that any (smooth) stationary $M^n \subset \R^{n+m}$ is contained in an affine subspace of dimension $\leq c(n) \Theta_M(\infty)$ for some constant $c(n)$.

Theorem \ref{thm:main} implies that density $2$ at infinity is strictly stronger than genus $0$, as one can easily construct examples of genus $0$ holomorphic curves in $\C^k$ which are not contained in a copy of $\C^2$, such as the image of $z \mapsto (z, z^2, z^3) \in \C^3$.  Theorem \ref{thm:main} also crucially relies on more than just stationarity of $T$.  For example, if $F : \C \setminus \{0\} \to \R^5$ is the map
\[
F(z) = \mathrm{Re} \left( -(z^{-1} + z/4), -\sqrt{-1}(z^{-1} - z/4), z, \sqrt{-1}z, \log(z) \right),
\]
then by \cite[Proposition 6.6]{hoffman-osserman} (see also \cite[Example 6.2]{Hojoo}) $F$ defines a complete, embedded, minimal annulus $M^2 \subset \R^5$ (with total curvature $-4\pi$) which is not contained in any affine $\R^4$, and has density $\Theta_M(\infty) = 2$.

\vspace{3mm}

There are effectively two distinct cases in Theorem \ref{thm:main}.  The first is when the tangent cone $\bC$ of $T$ at infinity is a multiplicity-two plane, and the second is when $\bC$ has multiplicity one.  In the latter case, our proof extends to any density.
\begin{theorem}\label{thm:main-mult-one}
Let $T$ be an area-minimizing $2$-current in $\R^{2+m}$ satisfying $\del T = 0$ and $\Theta := \Theta_T(\infty) < \infty$.  Suppose the tangent cone at infinity $\bC = \sum_{i=1}^{\Theta} \llbracket P_i \rrbracket$ of $T$ has multiplicity-one (so, all the $P_i$ are pairwise transverse).

Then after a rigid motion in $\R^{2+m}$, the span $P_1 + \cdots + P_\Theta$ can be identified with $\C^k \equiv \R^{2k} \times \{ 0^{2+m-2k} \}$ for some $k \geq 2$, and $T$ is a finite sum of irreducible, multiplicity-one algebraic varieties in affine copies of $\C^k$.  Precisely, there are irreducible algebraic curves $V_1, \ldots, V_l \subset \C^k$ and points $p_1, \ldots, p_l \in \R^{2+m}$, so that
\begin{equation}\label{eqn:main-mult-one-concl}
T = \sum_{j=1}^l \llbracket p_j + V_j \rrbracket \quad\text{and} \quad \sum_{j=1}^l \deg(V_j) = \Theta.
\end{equation}
\end{theorem}

When $T$ is indecomposable then of course $l = 1$, but we highlight that even if $T$ is decomposable (which is equivalent in this case to $\reg\, T$ being disconnected), then each decomposable piece of $T$ must be holomorphic with respect to the same orthogonal complex structure.

It is tempting to conjecture that \emph{any} area-minimizing 2-surface in $\R^n$ with quadratic area growth must be algebraic, as is the case in $\R^4$ (see \cite{Edelen_Franco_Minter}).  We believe the primary obstruction is understanding the structure and precise asymptotics of such a surface near infinity.

We mention that our methods can also apply to higher multiplicity planes in certain cases.  For general density $\Theta_T(\infty) = Q\in \N$, we are able to classify the ``simplest'' non-planar $T$ which are asymptotic to a single plane $P$ with multiplicity $Q$. Indeed, in this situation we are able to show that it is possible to adapt the planar frequency function of Krummel--Wickramasekera to analyze the behavior of the current, \emph{but now at infinity}. This allows us to associate another quantity which measures the ``complexity'' of $T$, namely the \emph{planar frequency} value $N_{T,P}(\infty)$ of $T$ at infinity, which we show takes values $N_{T,P}(\infty)\in \{p/q: p\in \{1,\dotsc,q-1\}, q\in\{2,\dotsc,Q\}\}$. In the ``simplest'' case $N_{T,P}(\infty)=1/Q$ (which, when $Q = 2$, is the only possible option), we can we can show that, up to rigid motion and dilation, $T$ must be contained in an affine copy of $\C^2$ and coincide with the variety
\[
\{(z,w)\in \C^2: z = p(w) )\}, \quad \deg(p) = Q
\]
and in particular $T$ must be embedded with genus $0$. Unlike when $Q=2$, in the case $Q\geq 3$ a priori $T$ could have interior singularities, including branch points, which we a posteriori show do not exist.

\vspace{3mm}

\textbf{Outline of the Proof:} 
The overall strategy is to show that $T$ is asymptotic to a holomorphic model surface $T' \subset \C^k$ at a sufficiently fast rate (Step 1), and then use the asymptotics and ``holomorphicity at infinity'' to show $T$ also lies inside $\C^k$ and is itself calibrated by the standard Kahler form (Step 2).  The second step is an application of a more general principle (Theorem \ref{thm:calibrated-teaser}) concerning minimizing $n$-currents $T$ which are ``calibrated at infinity.''

An initial observation is that from \cite{Riviere_2004} we know $T$ scale-invariantly decays (meaning $O(r^{1-\eps})$) to some unique tangent cone $\bC = \sum_{i=1}^\Theta \llbracket P_i\rrbracket$ at infinity.  To first order, the holomorphicity at infinity is provided by \cite{Morgan_1982}, which says that $\bC$ is holomorphic up to real rigid motion in the span of it's 2-planes $P_i$.  When $\bC$ is multiplicity-one this suffices, but when $\bC$ is a multiplicity-two plane $2 \llbracket P \rrbracket$ we must go to higher order to see the correct holomorphic structure, and in Step 1 we demonstrate that for large scales $T$ is modeled on the graph of a $1/2$-homogeneous $2$-valued Dirichlet-minimizing function on $P$.  In this case holomorphicity is provided by \cite{micallef-white}, which says that such a function is holomorphic up to rigid motion in $P^\perp$.  We remark that both \cite{Morgan_1982, micallef-white} prove holomorphicity by a direct variational argument.

\vspace{3mm}

\emph{Step 1.}  The first key step is to show $T$ has good asymptotic behavior.  Ultimately we wish to show that $T$ decays like $O(r^{-\eps})$ (or even just $O(1)$) to a holomorphic model surface/cone, but to do this we require a fairly precise graphical expansion near infinity.  The difficult case is when $T$ is asymptotic to a multiplicity-$2$ plane $2\llbracket P \rrbracket$, as there is not in general an Allard-type regularity theorem for area-minimizing surfaces of codimension $\geq 2$ near a plane with multiplicity (and indeed as previously mentioned a major difficulty is that a priori such $T$ might exhibit infinity topology drifting off to infinity).  Nonetheless, we are able to prove a multiplicity-2 (and even multiplicity-3) sheeting \emph{does} hold in this setting (Lemma \ref{lem:freq-sheeting}).

The primary tool for our sheeting theorem is an adaptation of the intrinsic planar frequency function $N_{T,P}(\rho)$ of Krummel--Wickramasekera \cite{KW23a}, which is defined for any $T$, but resembles the standard frequency function $N_v(\rho)$ when $T$ is well-approximated by a Dirichlet-minimizer.  We show that if both the area density $\Theta_T(\rho)$ and the planar frequency $N_{T,P}(\rho)$ are suitably pinched, then, outside of some ball, $T$ splits into a smooth, multiplicity-one, $2$-valued graph modeled on a $1/2$-homogeneous Dirichlet-minimizer. We state this below (see Sections \ref{sec:prelim} \& \ref{sec:freq} for details on notation).

\begin{lemma}[Multiplicity-2 sheeting; Lemma \ref{lem:freq-sheeting}]\label{lem:freq-sheeting-teaser}
    Let $P = \R^2\times\{0^m\}$ and let $T^2$ be an area-minimizing integral $2$-current in the cylinder $C_2(0)\subset \R^{2+m}$ over $P$, with $\del T = 0$. Take $\eps>0$ and $N_0\in [\eps,1-\eps]$. Suppose that
    \[
    E_{T,P}(0,2) \equiv \int_{C_2(0)}d(x,P)^2\, d\|T\| < \delta^2,
    \]
    and, for all $\rho\in [1/2,2]$,
    \[
    |\Theta_T(0,\rho)-2|\leq\delta \quad \text{and}\quad |N_{T,P}(\rho)-N_0|\leq \delta.
    \]
    Then, if $\delta = \delta(m,\eps)\in (0,1)$ is sufficiently small, we have $|N_0-1/2|<\eps$, and we can write
    \[
    T\res (B^2_1\setminus \overline{B^2_{1/2}}\times\R^m) = \llbracket \graph_{2\R^2}(v)\rrbracket,
    \]
    where $v$ is a smooth, 2-valued, solution to the minimal surface system, which is $\eps E_{T,P}(0,2)^{1/2}$-close to a $1/2$-homogeneous Dirichlet-minimizer.
\end{lemma}
For $T$ close to a plane of arbitrary multiplicity, the analogue of Lemma \ref{lem:freq-sheeting-teaser} implies a separation-type theorem -- see Remark \ref{rem:freq-sheeting}.

To apply Lemma \ref{lem:freq-sheeting-teaser}, we generalize the frequency monotonicity formula of \cite{KW23a} (originally used for $T$ decaying to $P$ near $0$) to show that the planar frequency $N_{T,P}(\rho)$ is monotone almost-increasing \emph{near infinity} (Theorem \ref{thm:freq-mono}). In fact, we prove a general scale-invariant differential inequality for $N_{T,P}(\rho)$ in terms of the $L^2$ height excess $E_{T,P}(0,\rho)$ of $T$ relative to $P$, which implies almost-monotonicity in a variety of potential cases. The upshot is that for an area-minimizing $T$ asymptotic to $2\llbracket P\rrbracket$, we can show that, for large scales, $T$ looks like the graph over $P$ of a function like
\begin{equation}\label{eqn:intro-1}
    v = A_1 r^{1/2}\cos(\theta/2)+B_1 r^{1/2}\sin(\theta/2) + A_0 \log(r) + B_0 + O(r^{-1/2}),
\end{equation}
for $A_1, B_1, A_0, B_0 \in P^\perp$ and $A_1, B_1$ linearly independent. Here, $(r,\theta)$ are polar coordinates in the $2$-plane $P$. We stress that almost-monotonicity of planar frequency near infinity is the mechanism which rules out pathological behavior of $T$ at large scales, as it allows one to show that the asymptotic behavior of $T$ is modeled on its tangent maps at infinity (which, at this stage of the proof, are not known to be unique), which are $2$-valued Dirichlet-minimizers that are \emph{non-zero}, \emph{average-free}, and \emph{homogeneous} (of degree $<1$).

To see how the asymptotics \eqref{eqn:intro-1} imply decay towards a holomorphic model we require a further refinement.  First, we note that the top-order $1/2$-homogeneous part
\begin{equation}\label{eqn:intro-1.5}
v_1 := A_1 r^{1/2} \cos(\theta/2) + B_1 r^{1/2} \sin(\theta/2)
\end{equation}
is itself a Dirichlet minimizing $2$-valued function, as it can be realized as the (unique) tangent map at infinity of $T$ over $P$.  It then follows from \cite{micallef-white} that $v_1$ is holomorphic up to rigid motion in $P^\perp$, i.e. $|A_1| = |B_1| =: a$ and $\< A_1,  B_1\> = 0$.  In particular, after rigid motion the $2$-valued graph of $v_1$ will coincide with the surface $\{ (z, w) \in \C^2 : w^2 = a^2 z \}$.  This is our holomorphic model.

Second, we use the area-minimizing property and a comparison argument to show that necessarily $A_0 = 0$, i.e.~the $\log(r)$ term in \eqref{eqn:intro-1} vanishes.  The intuition is that this term behaves like a catenoidal end, which is never minimizing, as at large scales it's more area-efficient to replace the catenoid with two large disks.  Trivially we can also get rid of $B_0$ by translating.  Altogether we deduce that $T$ decays to a the holomorphic surface $\{ w^2 = a^2 z\}$ like $O(r^{-1/2})$, which is the desired outcome of Step 1.

\vspace{3mm}

The easier case in Step 1 is when the tangent cone $\bC = \sum_{i=1}^\Theta \llbracket P_i \rrbracket$ at infinity is multiplicity-one.  Here we can apply Allard regularity directly to get that $T$ is asymptotically graphical over each plane $P_i$ by a function $u_i$ satisfying
\begin{equation}\label{eqn:intro-2}
u_i = A_{i0} \log(r) + B_{i0} + O(r^{-1}).
\end{equation}
Our holomorphic model is $\bC$ itself, as \cite{Morgan_1982} showed $\bC$ is holomorphic in the span of the $P_i$, after a suitable rigid motion.

As before, we can show each $A_{i0} = 0$ using the area-minimizing property, by constructing a competitor that has less area when some $A_{i0} \neq 0$.  When $\Theta = 2$ we can also get rid of the $B_{i0}$ by translation, but for general densities this may not be possible, and the best (a priori) we can hope for is that $T$ decays like $O(1)$ towards $\bC$.

\vspace{3mm}

\emph{Step 2:} The second key step is a fairly general principle that says whenever $T$ decays sufficiently fast at infinity to some asymptotically-conical, calibrated $T'$, then $T$ is calibrated also and by the same form.  We state this precisely below.

\begin{theorem}[Calibration rigidity at infinity; Theorem \ref{thm:calibrated}]\label{thm:calibrated-teaser}
For $n \geq 2$, $m \geq 1$, let $T, T'$ be area-minimizing integral $n$-currents in $\R^{n+m}$ with $\del T = \del T' = 0$, and let $\omega$ be a calibration on $\R^{n+m}$.  Suppose $T'$ is calibrated by $\omega$, is smooth and multiplicity-one outside $\overline{B_1}$, and satisfies
\begin{equation*}
\Theta_{T'}(\infty) < \infty, \quad \sup_{|x| > 1} |x| |A_{T'}(x)| < \infty,
\end{equation*}
where $A_{T'}$ denotes the curvature of $\spt\, T'$.  Suppose $T$ is asymptotic to $T'$ in the following sense
\begin{gather*}
T \res (\R^{n+m} \setminus \overline{B_1}) = \llbracket \graph_{T'}(u) \setminus \overline{B_1} \rrbracket, \quad  |u(x)| \leq |x|^{-(n-2)/2-\alpha},
\end{gather*}
for $u \in C^2(T' \setminus \overline{B_{1/2}}, T'^\perp)$ and $\alpha > 0$.

Then $T$ is calibrated by $\omega$ also.  Moreover, if $T' = \bC$ is conical, then it suffices to assume that $T$ is merely stationary and $u$ admits the weaker bound $|u(x)| \leq C |x|^{-(n-2)/2}$ for some constant $C$.
\end{theorem}

We can apply Theorem \ref{thm:calibrated-teaser} get holomorphicity for the $T$ as considered in Step 1 as follows.  We first show that $T$ is contained in some affine copy of $\C^k \equiv W \subset \R^{2+m}$, where the candidate space is $W = P \oplus \R A_1 \oplus \R B_1$ when $\bC = 2\llbracket P \rrbracket$, and $W = P_1 + \cdots + P_\Theta$ when $\bC$ is multiplicity-one.  To see this, we use the expansions \eqref{eqn:intro-1}, \eqref{eqn:intro-2} to deduce the coordinate functions on $W^\perp$ are bounded (and harmonic), and hence by parabolicity of $T$ must be constant.  After a real rigid motion we can assume $\bC$ is holomorphic in $\C^k$, and (in the first case) $A_1 = a e_3$, $B_1 = a e_4$ for some $a \in \R$.

We then second recall that from Step 1, given the normalizations of the previous paragraph, that for $R \gg 1$
\[
T \res (\R^4 \setminus \overline{B_R}) = \llbracket \graph_{T'}(u) \setminus \overline{B_R} \rrbracket, \quad T' = \llbracket \{ w^2 = a^2 z \} \rrbracket, \quad u(x) = O(|x|^{-1/2})
\]
when $\bC$ is a multiplicity-two plane, while
\[
T \res (\R^4 \setminus \overline{B_R}) = \llbracket \graph_{\bC}(u) \setminus \overline{B_R} \rrbracket, \quad u(x) = O(1),
\]
when $\bC$ is multiplicity-one.  Taking $\omega$ to be the standard Kahler form on $\C^k$, we obtain the exact hypotheses of Theorem \ref{thm:calibrated-teaser}, and can thereby deduce $T$ is holomorphic in $\C^k$.  Algebraicity follows from Stoll's theorem \cite{stolzenberg}, or by completing $T$ to a holomorphic current in $\CP^k$ as in \cite{Edelen_Franco_Minter}.

\vspace{3mm}

Let us highlight some aspects of Theorem \ref{thm:calibrated-teaser}, particularly in contrast with \cite{micallef}.  First, the \emph{choice} of holomorphic structure for $T$ becomes fairly transparent, being the holomorphic structure for the asymptotic model $T'$.  When $\bC$ has multiplicity-one, this is effectively the rigid motion which makes the union of planes comprising $\bC$ holomorphic (due to \cite{Morgan_1982}).  When $\bC$ is a multiplicity-two plane, then this is the rigid motion in $\bC^\perp$ which makes the top-order Jacobi field \eqref{eqn:intro-1.5} holomorphic (due to \cite{micallef-white}).  In general, the tangent cone $\bC$ itself does not encode enough information to capture the holomorphic structure on $T$, and one must expand $T$ to higher-order to see which choice to take.  For example, when $\bC$ is a multiplicity-two plane, then $\bC$ can be positively holomorphic for two orthogonal complex strutures on $\C^2$ (c.f. \cite[Remark 0.6]{Edelen_Franco_Minter}), and the correct choice among these is encoded in the expansion \eqref{eqn:intro-1}.

Second, Theorem \ref{thm:calibrated-teaser} does not require any knowledge or structure of $T$ in a compact region, and is codimension agnostic.  Only the asymptotic behavior of $T$ is important, and in fact it would even suffice to know $T$ decays to $T'$ sufficiently fast in a weaker sense, such as the flat norm.  The power $r^{-(n-2)/2-\eps}$ is the threshold for the energy of a graph to be finite at infinity, which geometrically is telling us that $\bM(T \res B_R) = \bM(T' \res B_R) + o(1)$ as $R \to \infty$.  (By the monotonicity formula the \emph{normalized} areas $R^{-n} \bM(T \res B_R)$ are always comparable.)

\vspace{3mm}

The basic idea of Theorem \ref{thm:calibrated-teaser} is to show that $\omega$ better and better approximates the volume form on $T$ at large scales, i.e.
\[
\int_{B_R} |1 - \omega[T_x T]| d\|T\| \equiv \bM(T \res B_R) - (T \res B_R)(\omega) = o(1) \quad \forall R \gg 1.
\]
To do this, we use the decay rate $r^{-(n-2)/2-\eps}$ to construct surfaces $S_R$ connecting $T \res B_R$ with $T' \res B_{2R}$ and which have comparable mass to $T' \res A_{2R, R}$, 
\begin{equation}\label{eqn:intro-3}
\del (T \res B_R + S_R) = \del (T' \res B_{2R}), \quad \bM(S_R) = \bM(T' \res A_{2R, R}) + o(1),
\end{equation}
and then we can combine \eqref{eqn:intro-3} with the area-minimizing property of $T$ to compare the mass of $T$ with $T(\omega)$ in large balls:
\begin{align*}
\bM(T \res B_R) \geq (T \res B_R)(\omega) &= (T' \res B_{2R})(\omega) - S_R(\omega) \\
&\geq \bM(T' \res B_{2R}) - \bM(S_R) \\
&\geq \bM(T' \res B_R) - o(1) \\
&\geq \bM(T \res B_R) - o(1).
\end{align*}

When $T' = \bC$ is conical, a slightly more sophisticated argument allows for the weaker decay $O(r^{-(n-2)/2})$, which is important since the best decay to $\bC$ we get from Step 1 is $O(1)$.  By a careful computation using the Fourier expansion on $\bC$, instead of \eqref{eqn:intro-3} we get the mass comparison
\[
\bM(S_R) = \bM(T' \res A_{2R, R}) + D + o(1)
\]
for some constant $D$ independent of $R$, and using the sharp form of the monotonicity formula we get
\[
\bM(T \res B_R) = \bM(\bC \res B_R) - D + o(1)
\]
for the \emph{same} constant $D$, which means the same argument as above carries through (note that the area-minimizing property is then not used here other than stationarity).  The computations in this borderline case closely parallel the theory of strictly minimizing cones developed in \cite{hardt-simon}.

\vspace{3mm}

\textbf{Acknowledgments:} We thank G\'abor Sz\'ekelyhidi, Brian White, and Neshan Wickramasekera for inspiring conversations.  N.E. was supported by NSF grant DMS-2506700 and a Simons Foundation travel award. This research was conducted during the period P.M. served as a Clay Research Fellow.

\vspace{3mm}

\textbf{AI Disclosure:} A very simplified (and false) version of Theorem \ref{thm:calibrated-teaser} was suggested to us spring 2026 in a conversation with Gemini 3.1 Pro about \cite{Edelen_Franco_Minter}.  After a further conversation with Brian White in summer 2026 after the first version of this article appeared (which used \cite{micallef} in place of Theorem \ref{thm:calibrated-teaser} to prove holomorphicity), we were inspired to revisit the argument and found it could be fixed, generalized, and adapted to our setting.  We additionally used ChatGPT 5.4 Pro to find the references \cite{hoffman-osserman, Hojoo}, and to proofread our paper for errors.  LLMs were not otherwise used in the mathematics or preparation of this paper.

\section{Preliminaries}\label{sec:prelim}

We work with $n$-currents and $n$-varifolds in $\R^{n+m}$, with our main results concerning the case $n=2$ and several auxiliary results holding for general $n$. We make use of standard notation and concepts from geometric measure theory, referring the reader to \cite{simon:gmt} for general background, and \cite{Almgren_2000, DLS11, DeLellis_Spadaro_2014} for background on multi-valued functions and the theory of high-codimension area-minimizing currents. We will write $c$ for a generic constant $c\geq 1$, and $\eps,\delta,\eta$ for generic (positive) constants $<1$. Named constants, such as $c_1,c_2,\dotsc$, are fixed and do not change in value.

We identify $\R^n\cong \R^n\times\{0^m\}\subset \R^{n+m}$.  Given $x\in \R^{n+m}$, the standard open Euclidean ball of radius $\rho>0$ centered at $x$ is $B_\rho(x):= \{y\in \R^{n+m}:|y-x|<\rho\}$, and when $x\in \R^n$ we write $B_\rho^n(x):= B_\rho(x)\cap \R^n$ for the corresponding $n$-dimensional ball in $\R^n$. For $0<\sigma<\rho$ write $A_{\rho,\sigma}(x):= B_\rho(x)\setminus \overline{B_\sigma(x)}$ for the open annulus with inner radius $\sigma$ and outer radius $\rho$, and $A^n_{\rho,\sigma}(x):= A_{\rho,\sigma}(x)\cap \R^n$. Write $B_\rho := B_\rho(0)$, and define analogously $B_\rho^n$, $A_{\rho,\sigma}$, and $A^n_{\rho,\sigma}$. We will always write $d_H(A,B)$ for the Hausdorff distance between subsets $A,B\subset\R^{n+m}$, and $d_A(x)$ for the distance between $x\in \R^{n+m}$ and a subset $A\subset\R^{n+m}$.

Let $P$ be an oriented $n$-plane (typically we take $P = \R^n\times\{0^m\}$). Write $\pi_P$ for the orthogonal projection to $P$, and $\pi_P^\perp$ for the orthogonal projection to $P^\perp$. Given $z\in \R^{n+m}$, we write $d_{z+P}(x):= |\pi_P^\perp(x-z)|$ for the distance of $x$ to the affine plane $z+P$. We also define the open cylinder over $P$ of radius $\rho$ and center $z$ by
\[
C_\rho(z,P):= \{x\in \R^{n+m}:|\pi_P(x-z)|<\rho\}.
\]
For shorthand, we write $C_\rho(z) := C_\rho(z,\R^n\times\{0^m\})$ and $C_\rho := C_\rho(0)$.

\medskip
We will often consider the special case of a $Q$-valued function which lifts to a single-valued function on a $Q$-branched copy of $\R^2\setminus\{0\}$. As such, we introduce specialized notation for this setting. Let us define $QS^1:= [0,2\pi Q]/\sim$, where $\sim$ is the equivalence relation on $[0,2\pi Q]$ with the only extra relation being $0\sim 2\pi Q$. We then write $v:Q\R^2\to \R^m$ to mean a function $v:(0,\infty)\times QS^1\to \R^m$ (note that $v$ is not defined at $0\in \R^2$), and similarly $v:Q\R^2\cap A_{b,a}\to \R^m$ for a function $v:(a,b)\times QS^1\to \R^m$.

One can interpret $v:Q\R^2\to \R^m$ as a function defined on a $Q$-branched cover of $\R^2\setminus\{0\}$, and in particular in a neighborhood of any $(r,\theta)$ with $r\neq 0$, we can think of $v$ as being defined on an open subset of $\R^2$.  Away from the origin, we can therefore make sense of the standard Euclidean derivatives $D^k v$, and talk about $v$ solving a PDE of the form $G(D^2v,Dv,v)=0$ by asking $v$ to solve the PDE locally, as a single-valued function, on each branch.

Given $v : Q\R^2 \to \R^m$, write
\begin{equation}\label{eqn:prelim-1}
\graph_{Q\R^2}(v):= \{(r\cos(\theta),r\sin(\theta),v(r,\theta): r\in (0,\infty)),\, \theta\in QS^1\},
\end{equation}
and similarly define $\graph_{Q\R^2}(v)\cap A_{b,a}$ for $v : Q\R^2 \cap A_{b, a} \to \R^m$, except now the radial variable $r$ obeys $r\in (a,b)$. We always endow $\graph_{Q\R^2}(v)$ with the orientation induced by the graphing map \eqref{eqn:prelim-1}.  Write $Jv := \sqrt{ \det(\delta_{ij} + \< D_i v, D_j v\>)} : Q\R^2 \to \R$ for the Jacobian of the graphing map, and for shorthand write
\[
\int_{Q\R^2 \cap A_{b, a}} v \, dx := \int_{QS^1} \int_a^b v(r, \theta) \, r dr d\theta .
\]

We will be particularly interested in $v:Q\R^2\to \R^m$ which solves the minimal surface system
\begin{equation}\label{eqn:mss}
g^{ij}D_{ij}^2v = 0,
\end{equation}
where $g^{ij} = (g^{-1})_{ij}$ for $g = (g_{ij})_{ij}$ with $g_{ij} = \delta_{ij} + \langle D_iv,D_jv\rangle$. This can be recast as a perturbation of the Laplacian, namely
\begin{equation}\label{eqn:mss-rough}
\Delta v = F(Dv,D^2v) = Dv\star Dv\star D^2v
\end{equation}
where we use the shorthand
\[
p \star p \star q \equiv \sum_{\substack{i,j,k,l=1,2 \\\alpha, \beta, \gamma = 1, \ldots, m}} c_{\alpha \beta \gamma}^{ijkl}(p, q) p_i^\alpha p_j^\beta q_{kl}^\gamma,
\]
for analytic functions $c^{ijkl}_{\alpha\beta\gamma}(p, q)$

Given $f : (a, \infty) \times QS^1 \to \R^m$ (or simply $f: (a,\infty) \to \R^m$), we write $f = O(r^\alpha)$ to mean $|f(r,\cdot)|\leq Cr^\alpha$ for $r\gg 1$ and some constant $C$. Similarly, we write $f = O_k(r^\alpha)$ to mean $r^i D^i f = O(r^\alpha)$ for each $i\in\{0,1,\dotsc,k\}$, and $f = O_{k,\gamma}(r^\alpha)$ to mean $f = O_k(r^\alpha)$ and $r^{k+\gamma}[D^k f]_{\gamma;A_{2r,r}} = O(r^\alpha)$. Finally, write $f = O_{k,1-}(r^\alpha)$ if $f = O_{k,\gamma}(r^\alpha)$ for all $\gamma<1$, and $f = O(r^{\alpha+})$ if $f = O(r^{\alpha+\eps})$ for all $\eps>0$. Similarly, we define $O_k(r^{\alpha+})$, $O_{k,\gamma}(r^{\alpha+})$, and $O_{k,1-}(r^{\alpha+})$ analogously.

\subsection{Varifolds and currents}

An integral $n$-current $T$ in $\R^{n+m}$ is area-minimizing (or more precisely, mass minimizing) if for every $U \subset \subset \R^{n+m}$ we have
\begin{equation}\label{eqn:minz}
\bM(T \res U) \leq \bM((T + S) \res U)
\end{equation}
for every integer rectifiable $n$-current $S$ satisfying $\del S = 0$ and $\spt\, S \subset U$.  Here $\bM(T) \equiv \|T\|(\R^{n+m})$ denotes the mass of $T$.

If $V$ is a stationary integral $n$-varifold in $\R^{n+m}$, the density ratio $\Theta_V(x, r)$ on the ball $B_r(x)$ is defined to be
\[
\Theta_V(x, r) = \frac{\|V\|(B_r(x))}{\omega_n r^n},
\]
where $\omega_n$ is the $n$-volume of the $n$-dimensional unit ball.  The classical monotonicity formula says that $\Theta_V(x, r)$ is increasing in $r$, and is constant in $r$ if and only if $V$ is dilation-invariant centered at $x$. In particular, the quantities
\[
\Theta_V(x) := \lim_{r \to 0} \Theta_V(x, r), \quad \Theta_V(\infty) := \lim_{r \to \infty} \Theta_V(0, r)
\]
always exist (but the latter may be infinite).  Integrality implies $\Theta_V(x) \geq 1$ iff $x \in \spt\, \|V\|$.

Likewise, if $T$ is an area-minimizing integral $n$-current with $\del T = 0$, then the underlying integral $n$-varifold $|T|$ is stationary, and so one can define $\Theta_T := \Theta_{|T|}$.  A basic fact is that any dilation-invariant area-minimizing $2$-current in $\R^{2+m}$ is supported on a union of transverse planes, and therefore if $n = 2$ it holds that $\Theta_T(x), \Theta_T(\infty) \in \N \cup \{ \infty \}$.

We will make use of the following ``excess'' quantities.  Define
\[
E_{T, P}(z, \rho) := \rho^{-n-2} \int_{C_{\rho}(z, P)} d_{z + P}(x)^2\, d\|T\|(x)
\]
for the \emph{$L^2$ height excess} in the cylinder $C_\rho(z,P)$, and
\begin{align*}
\cE_{T, P}(z, \rho) 
& := \rho^{-n} \bM(T \res C_{\rho}(z, P)) - \rho^{-n} \bM(\pi_{P \#} (T \res C_{\rho}(z, P))) \\
&\equiv \frac{1}{2\rho^n} \int_{C_\rho(z, P)} |\vec{T} - \vec{P}|^2\, d\|T\| 
\end{align*}
for the \emph{mass excess} in the cylinder (with the second alternative form being true up to flipping the sign of the orientation on $P$, cf.~\cite[§5.3.1]{Federer_GMT}). Here $\vec{T}(x)$ is the orienting $n$-vector for the approximating tangent space $T_x T$ of $T$ at $x$, $\vec{P}$ the orienting $n$-vector for $P$, and $|\vec{T} - \vec{P}|$ is the norm induced on the exterior product.

Though we will not name them explicitly, we will also commonly use the $L^2$ height and tilt excesses in a ball $B_\rho(y)$ (relative to an affine plane $z+P$), defined respectively as
\[
\rho^{-n-2} \int_{B_\rho(y)} d_{z + P}(x)^2\, d\|T\|(x), \quad \rho^{-n} \int_{B_\rho(y)} \|\pi_T - \pi_P\|^2\, d\|T\|(x).
\]
Here $\pi_T|_x$ is the projection onto the approximate tangent space $T_x T$, and $\|\pi_T - \pi_P\|$ is the Frobenius norm.  The mass excess and tilt excess are both kinds of geometric $W^{1,2}$-norms, but the former ``sees'' orientation while the latter does not. We record the following elementary geometric inequality
\begin{equation*}\label{eqn:tilt-by-tilt}
\|\pi_T - \pi_P\| \leq c(n, m) |\vec{T} - \vec{P}|.
\end{equation*}

We now record well-known relationships concerning the respective excesses.

\begin{lemma}[$L^2$ controls $L^\infty$]\label{lem:Linfty-by-L2}
Let $V$ be a stationary integral $n$-varifold in $B_1$.  Then for any $\gamma \in (0, 1)$ we have
\begin{equation}\label{eqn:Linfty-by-L2-concl1}
\sup_{x \in \spt\|V\| \cap B_\gamma} d_P(x)^2 \leq c(n, m, \gamma) \int_{B_1} d_P^2\, d\|V\|,
\end{equation}
If instead $V$ is a stationary integral $n$-varifold in $C_1$, we have
\begin{equation*}\label{eqn:Linfty-by-L2-concl2}
\sup_{x \in \spt\|V\| \cap C_\gamma} d_P(x)^2 \leq c(n, m, \gamma) \int_{C_1} d_P^2\, d\|V\|,
\end{equation*}
and in particular, we have
\[
\spt\|V\| \cap C_\gamma \subset B_{\gamma+c(n, m,\gamma)E_{V, P}(0, 1)^{1/2}}.
\]
\end{lemma}

\begin{proof}
The first inequality is well-known, see e.g.~\cite[Theorem 7.5(6)]{Allard_1972}.  To see the second inequality, we break into two cases. If $x \in \spt\|V\| \cap C_\gamma$ satisfies $d_P(x) \leq \eta \leq 1-\gamma$ for suitably small $\eta = \eta(n,m,\gamma)\in (0,1)$, then we have $\spt\|V\| \cap C_\gamma \subset B_{(1+\gamma)/2}$, and hence we can apply \eqref{eqn:Linfty-by-L2-concl1} in $B_{(1+\gamma)/2}$ to get
\[
d_P(x)^2 \leq c(n, m, \gamma) \int_{B_1} d_P^2\, d\|V\| \leq c E_{V, P}(0, 1).
\]
On the other hand, if $d_P(x) \geq \eta$, then in the ball $B_{\eta/2}(x)$ we have $d_P \geq d_P(x)/2$ and by monotonicity $\|V\|(B_{\eta/2}(x)) \geq \omega_n (\eta/2)^n$, and so we can estimate
\[
(d_P(x)/2)^2 \omega_n (\eta/2)^n \leq \int_{B_{\eta/2}} d_P^2\, d\|V\| \leq E_{V, P}(0, 1). \qedhere
\]
\end{proof}

\begin{lemma}[$L^2$ height controls tilt]\label{lem:tilt-by-L2}
Let $V$ be a stationary rectifiable $n$-varifold in $B_1(0)$.  Then for any $z \in \R^{n+m}$ and $\gamma \in (0, 1)$ we have
\[
\int_{B_\gamma} \|\pi_T - \pi_P\|^2\, d\|V\| \leq c(n, \gamma) \int_{B_1} d_{z + P}^2\, d\|V\|.
\]
\end{lemma}

\begin{proof}
The proof is standard, see e.g.~\cite[Lemma 8.13]{Allard_1972}.
\end{proof}

The following lemma is the first relationship which relies on being area-minimizing.

\begin{lemma}[$L^2$ controls mass excess {\cite[Lemma 3.7]{KW23a}}]\label{lem:mass-by-L2}
Let $T$ be an area-minimizing integral $n$-current in $C_1(0)$ satisfying
\begin{equation*}\label{eqn:mass-by-L2-hyp}
(\del T) \res C_1(0) = 0, \quad \sup_{x \in \spt\|T\| \cap C_1(0)} d_P(x) \leq 1.
\end{equation*}
Then for any $\gamma \in (0, 1)$ we have
\begin{equation*}\label{eqn:mass-by-L2-concl1}
\cE_{T, P}(0, \gamma) \leq c(n, m, \gamma) E_{T, P}(0, 1) .
\end{equation*}
\end{lemma}

\begin{proof}
This is the content of \cite[Lemma 3.7]{KW23a} (cf.~\cite[Lemma 4.1]{DLMS23}). 
\end{proof}

\begin{lemma}[Zero projection implies zero mass]\label{lem:zero-mass}
Let $T$ be an area-minimizing integral $n$-current in $C_1(0)$ satisfying
\begin{gather}\label{eqn:zero-mass-hyp}
(\del T) \res C_1(0) = 0, \quad \Theta_T(0, 1) \leq \theta + \eps, \quad E_{T, P}(0, 1) \leq \eps^2.
\end{gather}
Then, provided $\eps = \eps(n, m, \theta)$ is sufficiently small, we have
\begin{equation}\label{eqn:zero-mass-concl}
\pi_{P \#} (T \res C_1) = q\llbracket B^n_1(0) \rrbracket \text{ for some $|q| \leq \theta$.}
\end{equation}
Moreover, if $q = 0$, then for any $\gamma \in (0, 1)$ we have $T \res C_\gamma(0) = 0$ provided $\eps = \eps(n, m, \theta, \gamma)$ is sufficiently small.
\end{lemma}

\begin{proof}
By the constancy theorem for currents we know that $\pi_{P\#}(T\res C_1) = q\llbracket B^n_1(0)\rrbracket$ for some integer $q\in \Z$.  Fix any $\gamma\in (0,1)$.  We know that, up to changing the orientation on $P$,
\[
\gamma^{-n}\bM(T \res C_\gamma) - \gamma^{-n}|q|\omega_n = \cE_{T,P}(0,\gamma) \geq 0
\]
and so $|q|\omega_n \leq \|T\|(C_\gamma)$. But by Lemma \ref{lem:Linfty-by-L2} and hypothesis \eqref{eqn:zero-mass-hyp}, provided $\eps(n, m, \gamma)$ is sufficiently small, we have
\begin{equation}\label{eqn:zero-1}
    \bM(T \res C_\gamma) \leq \bM(T \res  B_{\gamma+c\eps}) \leq (\gamma+c\eps)^n\Theta_T(0,1)\omega_n \leq \theta\omega_n + c(n,m,\theta)\eps.
\end{equation}
and so combining these inequalities we get that $|q|\leq \theta + c(n,m,\theta)\eps$, and hence as $q\in \Z$ if we choose $\eps =\eps(n,m,\theta)$ sufficiently small we get $|q|\leq\theta$.

For the second claim we argue by contradiction.  Suppose, for a fixed $\gamma \in (0, 1)$, there are sequences $\eps_i \to 0$, $T_i$ satisfying the hypotheses \eqref{eqn:zero-mass-hyp} (with $\eps_i, T_i$ in place of $\eps, T$) and
\[
\pi_{P\#}(T_i \res C_1) = 0,
\]
but for which $T_i \res C_\gamma \neq 0$.

From \eqref{eqn:zero-1}, we know that for any $\gamma' \in (0, 1)$ we have uniform mass bounds $\bM(T_i \res (C_{\gamma'}) \leq c(n, m, \theta)$ for all $i \gg 1$, and therefore after passing to a subsequence we can find an area-minimizing integral $n$-current $T'$ (without boundary) in $C_1$ so that $T_i \to T'$ as currents in $C_1$, $|T_i| \to |T|$ as varifolds in $C_1$, and $\spt\, \|T_i\| \to \spt\, \|T'\|$ in the local Hausdorff distance in $C_1$. But then $T'$ must satisfy
\[
E_{T', P}(0, 1) = 0, \quad \pi_{P\#}(T' \res C_1) = 0,
\]
which implies $T' \res C_1 = 0$.  But now from our contradiction hypotheses and Lemma \ref{lem:Linfty-by-L2} we can find $x_i \in \spt\, \|T_i\| \cap C_\gamma$ with $d_P(x_i) \to 0$, which means by the local Hausdorff convergence of supports there must be some $x \in \spt\, \|T'\| \cap C_{(1+\gamma)/2}$, which is a contradiction.
\end{proof}

We also require the following improved $L^2-L^\infty$ estimate for area-minimizers.
\begin{theorem}[Improved $L^2-L^\infty$ estimate]\label{thm:improved-Linfty-L2}
Let $Q \in \N$, $P = \R^n \times \{0^m\}$, $p_1, \ldots, p_N \in \R^m$, and write $\bm{\pi} = \cup_{i=1}^N (p_i + P)$ for the union of parallel planes.  Let $T$ be an area-minimizing integral $n$-current in $C_1(0)$ satisfying
\[
\bM(T \res C_1(0)) \leq (Q+\tfrac{1}{2})\omega_n, \quad \pi_{P\#}(T \res C_1(0)) = Q\llbracket B_1^n \rrbracket, 
\]
\[
E := \int_{C_1} d(y, \bm{\pi})^2\, d\|T\| \leq 1.
\]
Then for any $\gamma < 1$, we have
\[
\spt\|T\| \cap C_\gamma(0) \subset \{ x : d(x, \bm{\pi}) \leq c E^{1/2} \}.
\]
for $c = c(n, m, Q, N, \gamma) > 0$.
\end{theorem}

\begin{proof}
This follows from \cite[Theorem 3.2 \& Corollary 3.3]{DLMS23} or \cite[Theorem 3.11]{KW23b}.
\end{proof}

Lastly, we record the below Liouville-type theorem for stationary $2$-varifolds with quadratic mass growth contained in a slab.  This result is essentially a classical fact about bounded harmonic functions, but for the reader's convenience we give the proof in the context of stationary varifolds.
\begin{lemma}\label{lem:vector-rigidity}
    Let $V$ be a stationary $2$-varifold in $\R^{2+m}$ with $\Theta_V(\infty)<\infty$. If for some unit vector $e$ we have $\sup_{\spt\|V\|}|\<e,x\>|<\infty$, then $\nabla_V\<e,x\> =0$ at $\|V\|$-a.e.~$x$.
\end{lemma}

\begin{remark}
    Lemma \ref{lem:vector-rigidity} is false without the assumption $\Theta_V(\infty)<\infty$ (see \cite{JX}). It is also false if $V$ has dimension $\geq 3$, as the standard catenoid in $\R^{n+1}$ for $n\geq 3$ is contained within two parallel planes. In the present lemma, one can weaken the growth rate on $|\<e,x\>|$ to sub-logarithmic.
\end{remark}

\begin{proof}
    After a rotation there is no loss in generality in assuming that $e=e_1$, so that $\<e,x\> = x_1$ is the first coordinate. Take $\phi\in C^1_c(\R)$ and choose the vector field $X = \phi^2(|x|)x_1e_1$ in the first variation formula for $V$. Using that $\pi_{T_xV}(e_1) = \nabla x_1|_x$, we obtain
    \[
    \int\phi^2|\nabla x_1|^2\, d\|V\| = -\int2\phi\phi^\prime\big\langle \pi_{T_xV}(x/|x|), e_1\big\rangle x_1\, d\|V\|
    \]
    which therefore gives, by the Cauchy--Schwarz inequality,
    \[
    \int\phi^2|\nabla x_1|^2\, d\|V\| \leq 4\int(\phi^\prime)^2x_1^2\, d\|V\|.
    \]
    Let $\phi(r)$ be a (smooth approximation of) a log-cutoff function at radius $R$, i.e.
    \[
    \phi(r) = \begin{cases}
        1 & \text{if }r\leq R,\\
        1-\frac{\log(r)-\log(R)}{\log(R^2)-\log(R)} & \text{if }R<r<R^2,\\
        0 & \text{if }R^2\leq r.
    \end{cases}
    \]
    Then we get
    \begin{align*}
    \int_{B_R}|\nabla x_1|^2\, d\|V\| & \leq 4\int_{B_{R^2}\setminus B_R}\frac{1}{r^2(\log(R))^2}x_1^2\, d\|V\|\\
    &\leq \frac{4\sup_{\spt\|V\|}x_1^2}{(\log(R))^2}\int_{B_{R^2}\setminus B_R}\frac{1}{r^2}\, d\|V\|\\
    & \leq \frac{32\pi\Theta_V(\infty)\cdot\sup_{\spt\|V\|}|x_1|^2}{\log(R)}
    \end{align*}
    where the last inequality follows from a standard dyadic-scale trick, since $\|V\|(B_r)\leq \Theta_V(\infty)\pi r^2$ for all $r>0$. 
    Taking $R\to\infty$ then gives $\nabla x_1 = 0$ $\|V\|$-a.e., which completes the proof.
\end{proof}

\subsection{Multi-valued functions and Dirichlet-minimizers}

In this section we recall some standard facts about multi-valued functions which minimize the Dirichlet energy. Standard references are \cite{Almgren_2000, DLS11, DeLellis_Spadaro_2014}.

We write $\cA_Q(\R^m)$ for the space of unordered $Q$-tuples, identified by sums of $Q$ Dirac point masses in $\R^m$. We write $\cG$ for the corresponding Almgren metric on $\cA_Q(\R^m)$.  A $W^{1,2}$ function $v:B^n_1\to \cA_Q(\R^m)$ is \emph{Dirichlet-minimizing} if
\[
\int_{B_1} |Dv|^2 \leq \int_{B_1} |Du|^2
\]
for any $W^{1,2}$ function $u:B_1\to \cA_Q(\R^m)$ which obeys $u = v$ $\cL^n$-a.e.~on $B_1\setminus\Omega$ for some open $\Omega\subset\subset B^n_1$. It is a standard fact (see e.g.~\cite[Theorem 0.9]{DLS11}) that Dirichlet-minimizers are $C^{0,\alpha}$ Hölder continuous for some universal $\alpha = \alpha(n,Q)\in (0,1)$.

If $v:B^n_1\to \cA_Q(\R^m)$ is Dirichlet-minimizing, we call $x\in B_1$ a \emph{regular point} if there exists an open neighborhood $U_x\subset B_1$ of $x$ such that $v|_{U_x} = \sum^Q_{j=1}\llbracket v_j\rrbracket$ for some analytic functions $v_1,\dotsc,v_Q:U_x\to \R^m$ with either $v_i\equiv v_j$ or $v_i(y)\neq v_j(y)$ for all $y\in U_x$. We write $\reg(v)$ for the set of \emph{regular points} of $v$, with the \emph{singular set} then $\sing(v):= B_1\setminus \reg(v)$.

An important fact we will need is the following dimension bound on the singular set of a Dirichlet-minimizer.

\begin{theorem}\label{thm:singular-set-bound}
    Suppose $v\in W^{1,2}(B^n_1;\cA_Q(\R^m))$ is Dirichlet-minimizing. Then $\sing(v)$ has Hausdorff dimension $\leq n-2$. Moreover, if $n=2$, $\sing(v)$ consists of isolated points.
\end{theorem}

\begin{proof}
    See \cite[Theorem 0.11 \& Theorem 0.12]{DLS11}.
\end{proof}

The standard tool for analyzing a (non-zero) Dirichlet-minimizer $v:\Omega\to\cA_Q(\R^m)$ in an open domain $\Omega \subset \R^n$ is the \emph{frequency function} $N$. We give a definition following \cite[Section 3.6]{KW23a}. Fix $\phi:[0,\infty)\to \R$ a non-increasing Lipschitz function satisfying $\phi(s)\equiv 1$ for $s\in [0,1/2]$ and $\phi(s)\equiv 0$ for $s\in [1,\infty)$. For each $z\in \Omega$ and $\rho\in (0,\dist(z,\del\Omega))$, define
\[
D_{v,z}(\rho):=\rho^{2-n}\int|Dv|^2\phi(r/\rho) \qquad \text{and} \qquad H_{v,z}(\rho):= -\rho^{1-n}\int|v|^2\frac{1}{r}\phi^\prime(r/\rho),
\]
where $r = r(x):= |x-z|$. The frequency function is then
\[
N_{v,z}(\rho):= \frac{D_{v,z}(\rho)}{H_{v,z}(\rho)}.
\]
We note that letting $\phi$ decrease to the indicator function of the interval $[0,1]$, one recovers the ``usual'' frequency function as originally defined by Almgren. The key property of the frequency function is that $N_{v,z}(\rho)$ is monotone non-decreasing in $\rho$ for each $z$, and furthermore $N_{v,z}(\rho)$ is constant in $\rho$ if and only if $v$ is homogeneous of degree $N\equiv N_{v,z}(\rho)$ centered at $z$.  If $0 \in \Omega$, we will often use the shorthand $N_v(\rho) \equiv N_{v, 0}(\rho)$.

Multi-valued Dirichlet-minimizing functions are important because of Almgren's (strong) Lipschitz approximation theorem for area-minimizing currents.  Here we write $\mathbf{v}(u)$ for the natural current associated to the graph of a Lipschitz multi-valued function $u: B^n_r\to \cA_Q(\R^m)$.

\begin{theorem}\label{thm:Lipschitz-approx}
    Fix $Q\in \{1,2,\dotsc\}$, $P = \R^n \times \{0^m\}$, and let $T$ be an area-minimizing integral $n$-current in $C_{1}(0) \subset \R^{n+m}$ satisfying
    \begin{gather*}
    (\del T)\res C_1 = 0, \quad \sup_{x\in\spt\|T\|}d_{P}(x)<\infty, \quad \pi_{P\#}T = Q\llbracket B_1^n\rrbracket, \\ \cE :=\cE_{V,P}(0,1)<\eps^2.
    \end{gather*}
    Then provided $\eps(n, m, Q) > 0$ is sufficiently small, there exists a Lipschitz function $u:B^n_{1/2}\to \cA_Q(\R^m)$ and a closed set $K\subseteq B^n_{1/2}$ so that
    \begin{enumerate}
        \item $T \res (K\times \R^m) = \mathbf{v}(u) \res (K\times \R^m)$;
        \item $\Lip(u) \leq c\cE^{\gamma}$;
        \item $\cH^n(B^n_{1/2}\setminus K) + \|T\|((B^n_{1/2}\setminus K)\times\R^m) \leq c\cE^{1+\gamma}$.
    \end{enumerate}
    Here $\gamma = \gamma(n,m,Q)$, $c = c(n, m, Q)$ are positive constants.
\end{theorem}

\begin{proof}
    See e.g.~\cite[Theorem 3.8]{KW23a}.
\end{proof}

\subsection{Uniqueness of tangent cones}

A crucial tool for us is that at large scales a $2$-dimensional area-minimizing current $T$ with quadratic area growth looks close to a \emph{fixed} tangent cone (which is necessarily a union of transverse planes with multiplicity), and moreover $T$ decays to this tangent cone at infinity at a controlled rate. This is the content of the following.

\begin{theorem}[Uniqueness of tangent cones at infinity, \cite{Riviere_2004}]\label{thm:unique}
Let $T$ be an area-minimizing integral $2$-current in $\R^{2+m}$ with $\del T = 0$ and $\Theta := \Theta_T(\infty) < \infty$.  Then there are transverse $2$-planes $\{ P_1, \ldots, P_\ell \}$ and multiplicities $n_1, \ldots, n_\ell \in \N$ so that $\bC = n_1 \llbracket P_1 \rrbracket + \cdots + n_\ell \llbracket P_\ell \rrbracket$ is the unique tangent cone of $T$ at infinity.  In other words, $(1/R)_\# T \to \bC$ as currents and varifolds, and $(1/R) \spt\|T\| \to \spt\, \|\bC\|\equiv \cup_i P_i$ locally in the Hausdorff distance, as $R\to\infty$.

Moreover, there is an $\alpha = \alpha(m, \Theta) > 0$ so that for all $R \gg 1$ we have the scale-invariant decay:
\begin{gather}
R^{-1} d_H( \spt\|T\| \cap B_R, \spt\|\bC\| \cap B_R) \leq R^{-\alpha}, \label{eqn:unique-concl1} \\
R^{-n-2} \int_{B_R} \mathrm{dist}(x, \bC)^2\, d\|V\|(x) \leq R^{-\alpha}. \label{eqn:unique-concl2}
\end{gather}
\end{theorem}

\begin{proof}
This follows primarily from the reverse epiperimetric inequality shown in \cite{Riviere_2004}.  Specifically, \cite{Riviere_2004} implies there is an $\eps = \eps(m, \Theta) \in (0, 1)$ and $R_0 = R_0(T) \gg 1$ so that
\begin{equation}\label{eqn:unique-1}
\Theta - \Theta_T(0, R) \leq (R_0/R)^{2\eps}(\Theta - \Theta_T(0, R_0)) \quad \forall R \geq R_0.
\end{equation}
Arguing as in \cite{White_1983}, \eqref{eqn:unique-1} implies that given any sequence $R_i \to \infty$, the currents $\del (((1/R_i)_\# T) \res B_1)$ form a Cauchy sequence in the flat norm, which implies $(1/R)_\# T \to \bC$ as currents (for some area-minimizing cone $\bC$), and hence as varifolds by the minimizing property (cf.~Lemma \ref{lem:mass-by-L2} and Theorem \ref{thm:Lipschitz-approx}) and in support by the monotonicity formula.

To see the convergence rates \eqref{eqn:unique-concl1} and \eqref{eqn:unique-concl2}, we can use Lemma \ref{lem:L2-by-mono}. For any $\delta>0$, by enlarging $R_0$ one can assume that $d_H( \spt\|T\| \cap B_R, \spt\|\bC\| \cap B_R) \leq \delta R$ for all $R \geq R_0$. By picking $\delta = \delta(\bC)$ suitably small, we can find a smooth, $2$-homogeneous function $f$ on $\R^{2+m} \setminus \{0\}$ which agrees with the square of the distance to $\bC$ in the region
\[
\{ x \in \R^{2+m} \setminus \{0\} : \dist(x, \bC) \leq 2\delta |x| \}.
\]
Applying Lemma \ref{lem:L2-by-mono} to $f$ with the decay \eqref{eqn:unique-1}, we get
\begin{equation}\label{eqn:unique-2}
|e(R) - e(S)| \leq c(m, \Theta) (\log(R/S) + 1) (R_0/S)^{\eps} \quad \forall R_0 < S < R,
\end{equation}
where we write $e(R) := R^{-n-2} \int_{A_{R, R/2}} \dist(x, \bC)^2\, d\|V\|$. Taking $R_i = R_0 2^i$ we deduce from \eqref{eqn:unique-2} that
\[
e(R_i) = e(R_i) - \lim_{R \to \infty} e(R) \leq \sum_{j \geq i} |e(R_{j+1}) - e(R_j)| \leq c(m, \Theta) R_i^{-\eps},
\]
and hence
\begin{align}
\int_{B_{R_i}} \dist(x, \bC)^2 d\|V\| &\leq \int_{B_{R_0}} \dist(x, \bC)^2 d\|V\| + \sum_{j=1}^i e(R_j)R_j^{n+2} \nonumber \\
&\leq \int_{B_{R_0}} \dist(x, \bC)^2 d\|V\| + c R_i^{n+2-\eps} \nonumber \\
&\leq c(m, \Theta) R_i^{n+2-\eps} \label{eqn:unique-3}
\end{align}
provided $i \gg 1$.  Combining the above with \eqref{eqn:unique-2} and ensuring $R \gg 1$ we get \eqref{eqn:unique-concl1} with $\alpha = \eps/2$.  The decay estimate in \eqref{eqn:unique-concl1} then follows from \eqref{eqn:unique-concl2} and Lemma \ref{lem:Linfty-by-L2}.
\end{proof}

\section{Planar frequency for area-minimizing currents}\label{sec:freq}

In this section we give a generalized version of Krummel--Wickramasekera's planar frequency monotonicity formula for area-minimizing $n$-currents in $\R^{n+m}$, as originally introduced in \cite{KW23a}. For this, we fix $\phi:[0,\infty)\to \R$ the Lipschitz function with $\phi(s)\equiv 1$ for $s\in [0,1/2]$, $\phi(s) = 2-2s$ for $s\in [1/2,1]$, and $\phi(s)\equiv 0$ for $s\in [1,\infty)$. We also write $r(x):= |\pi_P(x)|$ for the cylindrical radius with respect to the plane $P := \R^n\times\{0^m\}$. Then define the quantities
\[
H_{T,P}(\rho):= -\rho^{1-n}\int d^2_P|\nabla r|^2\frac{1}{r}\phi^\prime(r/\rho)\, d\|T\|,
\]
\[
D_{T,P}(\rho) := \frac{1}{2}\rho^{2-n}\int\|\pi_T-\pi_P\|^2\phi(r/\rho)\, d\|T\|,
\]
and, when $H_{T,P}(\rho)>0$,
\[
N_{T,P}(\rho):= \frac{D_{T,P}(\rho)}{H_{T,P}(\rho)}.
\]
The function $N_{T,P}$ is an intrinsic notion of frequency for the current $T$ relative to the plane $P$ (at scale $\rho$), which effectively reduces to the `standard' frequency $N_v$ if $T$ is represented by the multi-graph of a `small' function $v$ over $P$.

The main theorem of this section is the following. The proof is a refinement of the calculations in the proof of the planar frequency monotonicity given in \cite{KW23a}, which we defer to Appendix \ref{sec:freq-mono-appendix}.

\begin{theorem}\label{thm:freq-mono}
    Let $\theta\geq 1$ and $P = \R^n\times\{0^m\}$. Then, there exists $\delta = \delta(n,m,\theta)\in (0,1)$ such that if $T$ is an area-minimizing integral $n$-current in $C_{2\rho}(0)$ obeying
    \[
    (\del T)\res C_{2\rho}(0) = 0,\qquad \Theta_T(0,\rho/2)\geq 1/2, \qquad \Theta_{T}(0,2\rho)\leq\theta,
    \]
    \[
    E_{T,P}(0,2\rho)\leq E \leq \delta^2,
    \]
    and such that $\spt\, \|T\|$ does not coincide with (any translate of) $P$ in $C_{2\rho}(0)$, then
    \begin{equation}\label{eqn:freq-mono-concl1}
\frac{d}{d\rho} \log N_{T, P}(\rho) \geq -c \rho^{-1} E^{\gamma} - c' E^{\gamma-\gamma_1} \frac{d}{d\rho}\left[ (\rho^{-2} D_{T, P}(\rho))^{\gamma_1} \right],
\end{equation}
and
\begin{equation}\label{eqn:freq-mono-concl2}
\left| \frac{d}{d\rho} \log H_{T, P}(\rho) - 2 \rho^{-1} N_{T, P}(\rho) \right| \leq c \rho^{-1} E^{\gamma}.
\end{equation}
Here, $0<c^\prime<c$ and $0<\gamma_1<\gamma$ are all constants depending only on $m,n,\theta$.
\end{theorem}
The primary corollary of Theorem \ref{thm:freq-mono} for us is almost-monotonicity of $N_{T,P}$ at large scales when $T$ decays polynomially to $P$ at infinity (almost-monotonicity of $N_{T,P}$ at small scales given polynomial decay to $P$ at a point was the original theorem of \cite{KW23a}). We remark that the differential inequality is more general, and would also imply an almost-monotonicity of $N_{T,P}$ if $E_{T,P}(\rho)$ were to decay logarithmically in $\rho$.

\begin{cor}\label{cor:freq-mono}
    Let $T$ be an area-minimizing integral $n$-current in $B_R\subset\R^{n+m}$ whose support does not coincide with a plane. Suppose that
    \[
    (\del T)\res B_R = 0,\qquad \Theta_T(0,1/2)\geq 1/2,\qquad \Theta_T(0,R)\leq \theta,
    \]
    and, for some fixed $\eta,\alpha>0$,
    \[
    \rho^{-n-2}\int_{B_\rho}d_P^2\, d\|T\| \leq \eta^2\rho^{-2\alpha} \qquad \text{for all }1\leq\rho\leq R.
    \]
    Then, provided $\eta = \eta(n,m,\theta)\in (0,1)$ is sufficiently small, there are constants $c^\prime<c$, $\gamma_1<\gamma$ depending only on $n,m,\theta,\alpha$ such that the function
    \begin{gather}\label{eqn:cor1-freq-mono-concl1}
    \rho\mapsto N_{T,P}(\rho) \exp(-c\eta^{2\gamma}\rho^{-2\alpha\gamma}+c^\prime(\eta^2 \rho^{-2\alpha})^{\gamma-\gamma_1}(\rho^{-2}D_{T,P}(\rho))^{\gamma_1})
    \end{gather}
    is non-decreasing for $1\leq\rho\leq R/4$.
\end{cor}

\begin{proof}
    This follows directly from \eqref{eqn:freq-mono-concl1}, since by Lemma \ref{lem:Linfty-by-L2} we have $E_{T,P}(\rho)\leq \eta^2\rho^{-2\alpha} \leq \eta^2$ for all $1\leq\rho\leq R/2$, provided $\eta(n, m, \theta)$ is small.  From Theorem \ref{thm:freq-mono} and Lemma \ref{lem:tilt-by-L2} we get
    \begin{align*}
    &\frac{d}{d\rho}\left( \log N_{T,P}(\rho) + c^\prime (\eta^2 \rho^{-2\alpha})^{\gamma - \gamma_1} (\rho^{-2} D_{T, P}(\rho))^{\gamma_1} \right) \\
    &\geq -c\eta^{2\gamma}\rho^{-2\alpha\gamma-1} - 2\alpha(\gamma - \gamma_1) c^\prime \eta^{2(\gamma - \gamma_1)} \rho^{-2\alpha(\gamma - \gamma_1) - 1} (\rho^{-2} D_{T, P}(\rho))^{\gamma_1} \\
    &\geq -c \eta^{2\gamma} \rho^{-2\alpha \gamma - 1} - c \eta^{2(\gamma - \gamma_1)} \rho^{-2\alpha(\gamma - \gamma_1) - 1} E_{T, P}(2\rho)^{\gamma_1} \\
    &\geq -c \eta^{2\gamma} \rho^{-2\alpha \gamma - 1}
    \end{align*}
    for all $1 \leq \rho \leq R/4$, which implies \eqref{eqn:cor1-freq-mono-concl1} is monotone.
\end{proof}

The second corollary is the standard growth and doubling estimates which follow from the almost monotonicity.  We also defer the proof to Appendix \ref{sec:freq-mono-appendix}.

\begin{cor}\label{cor:H-control}
    Let $T$ be an area-minimizing integral $n$-current in $C_{2\rho_2}$, for $0<\rho_1<\rho_2$, and suppose that
    \[
    (\del T)\res C_{2\rho_2} = 0,\qquad \Theta_T(0,\rho_1/2)\geq 1/2,\qquad \Theta_T(0,2\rho_2)\leq\theta,
    \]
    and
    \[
    E_{T,P}(2\rho)\leq \eps^2 \quad \text{and} \quad N_1\leq N_{T,P}(\rho)\leq N_2\qquad \text{for all }\rho_1\leq\rho\leq\rho_2.
    \]
    Then, provided $\eps = \eps(m,n,\theta)$ is sufficiently small, we have
    \begin{equation}\label{eqn:H-control-concl1}
(\rho_2/\rho_1)^{2N_1 - c \eps^{2\gamma}} \leq \frac{H_{T, P}(\rho_2)}{H_{T, P}(\rho_1)} \leq (\rho_2/\rho_1)^{2N_2 + c \eps^{2\gamma}}
\end{equation}
and
\begin{equation}\label{eqn:H-control-concl2}
E_{T, P}(\rho_2) \leq c (\rho_2/\rho_1)^{2N_2 + c \eps^{2\gamma}} E_{T, P}(\rho_1).
\end{equation}
where $c = c(n,m,\theta)\in (0,\infty)$.
\end{cor}

\section{Refined asymptotics}

In this section we prove an asymptotic expansion for $u:Q\R^2\to \R^m$ solving the minimal surface system. We will make use of the Fourier expansion for functions $u:(a,b)\times QS^1\to \R$. Indeed, for any such function, and for each $r$, we can expand $u(r,\cdot)$ in terms of the standard orthogonal basis of $L^2(QS^1)$:
\[
u(r,\theta) = a_0(r) + \sum^\infty_{k=1} \big(a_k(r)\cos(k\theta/Q) + b_k(r)\sin(k\theta/Q)\big),
\]
where
\[
a_0(r) = \frac{1}{2\pi Q}\int^{2\pi Q}_0 u(r,\theta)\, d\theta, \quad a_k(r) = \frac{1}{\pi Q}\int^{2\pi Q}_0 u(r,\theta)\cos(k\theta/Q)\, d\theta, 
\]
\[
b_k(r) = \frac{1}{\pi Q}\int^{2\pi Q}_0u(r,\theta)\sin(k\theta/Q)\, d\theta.
\]
Thus,
\[
\frac{1}{\pi Q} \int^{2\pi Q}_0u(r,\theta)^2\, d\theta = 2 a_0(r)^2+ \sum_{k\geq 1}\left(a_k(r)^2+b_k(r)^2\right).
\]
We note that if $u:(1,\infty)\times QS^1\to \R^m$ and $u = O(r^{\alpha+})$, then from the above expressions we see that $a_k(r) = O(r^{\alpha+})$ and $b_k(r) = O(r^{\alpha+})$ for each $k$ also.

If $\Delta u = 0$, then by considering the corresponding ODEs solved by $a_k(r), b_k(r)$ we may expand
\begin{align}
u = A_0 + B_0\log(r) + \sum^\infty_{k=1}&\left(A_kr^{k/Q}+A_{-k}r^{-k/Q}\right)\cos(k\theta/Q)\nonumber\\
& + \sum_{k=1}^\infty\left(B_kr^{k/Q}+B_{-k}r^{-k/Q}\right)\sin(k\theta/Q), \label{eqn:q-harmonic-fourier}
\end{align}
for constants $A_k,B_k\in \R^m$ so that, for each $r$,
\begin{align}
    \frac{1}{\pi Q}\int_{QS^1}u(r,\theta)^2\, d\theta = 2(A_0+B_0\log(r))^2+\sum^\infty_{k=1}&\left(A_kr^{k/Q}+A_{-k}r^{-k/Q}\right)^2\nonumber\\
    & + \sum_{k=1}^\infty \left(B_k r^{k/Q}+B_{-k}r^{-k/Q}\right)^2. \label{eqn:q-harmonic-L2}
\end{align}
From \eqref{eqn:q-harmonic-L2} it follows that if $u$ is harmonic and $u = O(r^{\alpha+})$ then $A_k,B_k=0$ for $k/Q>\alpha$. Conversely, if $A_k = B_k =0$ for $k>k_0$, then $u = O(r^{k_0/Q+})$.

\medskip

Our first lemma extends the above analysis to certain Poisson equations.

\begin{lemma}\label{lem:ps}
    Take $Q\in \N$, and $\alpha>0$. Suppose $f:[1,\infty)\times QS^1\to \R^m$ is a smooth function satisfying $f = O(r^{-1-\alpha+})$. Then, there is a smooth function $u:(1,\infty)\times QS^1\to \R^m$ solving
    \[
    \Delta u = f \quad \text{with} \quad u = O_{1,1-}(r^{1-\alpha+}).
    \]
    In fact if $f = O_{k,\gamma}(r^{-1-\alpha+})$ for some $\gamma\in (0,1)$, then $u = O_{k+2,\gamma}(r^{1-\alpha+})$.
\end{lemma}

\begin{proof}
    We shall construct a solution in terms of its complex Fourier expansion
    \begin{equation}\label{eqn:ps-1}
        u(r,\theta) = \sum_{k\in\Z}a_k(r)e^{\sqrt{-1}k\theta/Q},
    \end{equation}
    for $a_k(r)$ to be defined below.  Set
    \[
    f_k(r):=\frac{1}{2\pi Q}\int_0^{2\pi Q}f(r,\theta)e^{-\sqrt{-1}k\theta/Q}\, d\theta ,
    \]
    and note that $f_k = O(r^{-1-\alpha+})$. (We will slightly abuse notation here as $a_k(r),f_k(r)\in\C^m$.)

    We then have $\Delta u = f$ only if each coefficient $a_k$ solves the ODE
    \[
    a_k^{\prime\prime}+r^{-1}a_k^\prime - r^{-2}(k/Q)^2a_k = f_k,
    \]
    which can be written in divergence form as
    \begin{equation}\label{eqn:ps-2}
    r^{k/Q-1}\big(r^{1-2k/Q}(r^{k/Q}a_k)^\prime\big)^\prime = f_k.
    \end{equation}
    Since $\alpha>0$ we can find explicit solutions $a_k(r)$ of \eqref{eqn:ps-2} of the form
    \[
    a_k(r) = r^{-|k|/Q}\int_{r_k}^r \left(s^{2|k|/Q-1}\int^s_{s_k}t^{1-|k|/Q}f_k(t)\, dt\right)\, ds,
    \]
    where we set $s_k=\infty$ if $|k|/Q>1-\alpha$ and $s_k=1$ if $|k|/Q\leq 1-\alpha$, and $r_k = \infty$ if $|k|/Q < \alpha - 1$ and $r_k = 1$ if $|k|/Q \geq \alpha - 1$.

  We now claim that these choices of $a_k$ in \eqref{eqn:ps-1} give the desired solution.  For simplicity let us only prove the case $\alpha \in (0, 1]$, the case when $\alpha > 1$ is similar. First, let us fix $\beta\in (0,\alpha)$ so that $\{k\in \Z:|k|/Q>1-\alpha\} = \{k\in \Z:|k|/Q>1-\beta\}$. Then, observe that by our assumption of $f = O(r^{-1-\alpha+})$, for $\eps = \eps(\alpha,\beta)\in (0,1)$ sufficiently small, we have
    \[
    \vertiii{f}_\beta^2:=\int^\infty_1 t^{1+2\beta}\|f(t,\cdot)\|^2_{L^2(QS^1)}\, dt \leq c(f,\eps)\int^\infty_1 t^{-1+2(\beta-\alpha+\eps)}\, dt<\infty.
    \]
    We now compute bounds on $a_k(r)$ as follows. If $|k|/Q>1-\alpha$ (and hence $|k|/Q>1-\beta$) and $r\geq 1$, we have:
    \begin{align*}
        |a_k(r)| & \leq r^{-|k|/Q}\int^r_1\left( s^{2|k|/Q-1}\int^\infty_1 t^{1-|k|/Q}|f_k(t)|\, dt\right)\, ds\\
        & \leq r^{-|k|/Q}\int^r_1 s^{2|k|/Q-1}\left(\int^\infty_1 t^{1+2\beta}f_k(t)^2\, dt\right)^{1/2}\left(\int^\infty_s t^{1-2|k|/Q-2\beta}\, dt\right)^{1/2}\, ds\\
        & \leq \left(\int^\infty_1 t^{1+2\beta}f_k(t)^2\, dt \right)^{1/2}\cdot r^{-|k|/Q}\int^r_1 s^{2|k|/Q-1}\left(\frac{s^{2-2|k|/Q-2\beta}}{2|k|/Q-2+2\beta}\right)^{1/2}\, ds\\
        & \leq \frac{1}{\sqrt{2}}\left(\int^\infty_1 t^{1+2\beta}f_k(t)^2\, dt \right)^{1/2}\frac{r^{1-\beta}-r^{-|k|/Q}}{(|k|/Q+(1-\beta))(|k|/Q-(1-\beta))^{1/2}}.
    \end{align*}
    The above computation was justified because $2-2|k|/Q-2\beta<0$. We deduce that for every such $k$, we have
    \begin{equation}\label{eqn:ps-3.5}
        |a_k(r)|^2 \leq c(\beta,\alpha,Q)r^{2-2\beta}\int^\infty_1 t^{1+2\beta}f_k(t)^2\, dt \qquad \text{for all }r\geq 1.
    \end{equation}
    On the other hand, if $|k|/Q\leq 1-\alpha$, then we can compute similarly
    \begin{align*}
        |a_k(r)| & \leq r^{-|k|/Q}\int^r_1 s^{2|k|/Q-1}\left(\int^\infty_1 t^{1+2\beta}f_k(t)^2\, dt\right)^{1/2}\left(\int^s_1 t^{1-2|k|/Q-2\beta}\, dt\right)^{1/2}\, ds\\
        & \leq \left(\int^\infty_1 t^{1+2\beta}f_k(t)^2\, dt\right)^{1/2} r^{-|k|/Q}\int^r_1 s^{2|k|/Q-1}\left(\frac{s^{2-2|k|/Q-2\beta}-1}{2-2|k|/Q-2\beta}\right)^{1/2}\, ds\\
        & \leq \frac{1}{\sqrt{2}}\left(\int^\infty_1 t^{1+2\beta}f_k(t)^2\, dt\right)^{1/2}\frac{r^{1-\beta}-r^{-|k|/Q}}{(|k|/Q+(1-\beta))(1-\beta-|k|/Q)^{1/2}},
    \end{align*}
    which is justified because $1-|k|/Q-\beta \geq \alpha-\beta>0$. We deduce that \eqref{eqn:ps-3.5} also holds for every such $k$, and therefore
    \[
    \sum_k |a_k(r)|^2 \leq c(\alpha,\beta,Q)r^{2-2\beta}\int^\infty_1 t^{1+2\beta}\|f(t,\cdot)\|_{L^2(QS^1)}^2\, dt \quad\text{for all }r\geq 1.
    \]
    It follows that the sum in \eqref{eqn:ps-1} converges in $L^2$ to an $L^2$ function $u:(1,\infty)\times QS^1\to \R^m$ which weakly solves the equation $\Delta u = f$ and admits the $L^2$ bound
    \begin{equation}\label{eqn:ps-4}
    \|u(r,\cdot)\|_{L^2(QS^1)}\leq c(\alpha,\beta,Q)r^{1-\beta}\vertiii{f}_\beta \quad \text{for all }r\geq 1\text{ and }0<\beta<\alpha.
    \end{equation}
    From standard elliptic PDE theory, $u$ is smooth and solves $\Delta u = f$ in the classical sense. From \eqref{eqn:ps-4} we have $u = O(r^{1-\alpha+})$, and so the asymptotic $u = O_{1,1-}(r^{1-\alpha+})$ follows from $C^{1,\alpha}$ Schauder theory. The further asymptotics $u = O_{k+2,\gamma}(r^{1-\alpha+})$ when $f = O_{k,\gamma}(r^{-1-\alpha+})$ follow from $C^{2,\alpha}$ Schauder theory.
\end{proof}

The next two lemmas are the main results of this section.

\begin{lemma}[Refined asymptotics near a multiplicity $Q$ plane]\label{lem:refined-asymptotics}
    Take $Q \in \N$, $\alpha \in (0, 1]$. Suppose $v:(1,\infty)\times QS^1\to \R^m$ is a $C^2$ function solving the minimal surface system \eqref{eqn:mss} which admits the decay $v = O(r^{1-\alpha+})$, $Dv = o(1)$. Let $\alpha_Q:= \max\{k/Q: k/Q\leq1-\alpha,\ k\in\Z\}$.

    Then, we can find $A_k,B_k\in \R^m$ and $r_0>1$ so that, for $r\geq r_0$, $v$ admits the expansion
    \begin{align}
        v(r,\theta) = & \sum_{0<k/Q\leq\alpha_Q}\big(A_kr^{k/Q}\cos(k\theta/Q) + B_kr^{k/Q}\sin(k\theta/Q)\big) \label{eqn:refined-concl1}\\
        & \hspace{1em} + A_0\log(r) + B_0 \nonumber \\
        & \hspace{2em} + \sum_{3\alpha_Q-2<k/Q<0}\big(A_kr^{k/Q}\cos(k\theta/Q)+B_kr^{k/Q}\sin(k\theta/Q)\big) \nonumber \\
        & \hspace{3em} + O_{2,1-}(r^{3\alpha_Q-2+}). \nonumber
    \end{align}
    In particular, if $Q=2$ we have the expansion
    \[
    v(r,\theta) = A_1 r^{1/2}\cos(\theta/2) + B_1r^{1/2}\sin(\theta/2) + A_0\log(r) + B_0 + O_2(r^{-1/2+}),
    \]
    and if $Q=1$ we have the expansion
    \[
    v(r,\theta) = A_0\log(r) + B_0 + \frac{A_{-1}}{r}\cos(\theta) + \frac{B_{-1}}{r}\sin(\theta) + O_2(r^{-2+}).
    \]
\end{lemma}

\begin{proof}
    Let $\alpha_Q$ be the smallest number $\in [0,1-\alpha]$ for which $v = O(r^{\alpha_Q+})$.  If $r\gg 1$, we can assume that $|v(r,\theta)|/r + |Dv(r, \theta)| \leq\eta$, for some $\eta = \eta(n,m)$ suitably small. By applying Allard's regularity theorem to the single-sheeting surfaces $\{(r\cos(\theta),r\sin(\theta),v(r,\theta): r>0,\ |\theta-\theta_0|<\pi/2\}$ for suitable choices of $\theta_0$, we get the a priori bounds
    \begin{align}
        |Dv|_{C^{0}(\{R<r<2R\})} + R^\gamma[Dv]_{\gamma;\{R<r<2R\}} & \leq c(m,\gamma)R^{-1}|v|_{C^0(\{R/2<r<4R\})} \nonumber \\
        & \leq c(m,\gamma,\eps)\eta  \R^{-1+\alpha_Q+\eps} \label{eqn:refined-1}
    \end{align}
    for any $\gamma\in (0,1)$, any $\eps>0$, and all $R\gg 1$.\footnote{One could remove the hypothesis on $Dv$ and deduce \eqref{eqn:refined-1} directly from $v = O(r^{1-\alpha+})$ by using the viscosity-based regularity theory for minimal surface system due to \cite{savin-allard}.}  We can then apply Schauder theory to the minimal surface system \eqref{eqn:mss} to get
    \[
    R^{k-1}|D^kv|_{C^0(\{R<r<2R\})} \leq c(m,\alpha,\eps,k)R^{-1+\alpha_Q+\eps} \qquad \text{for all }k\geq 2.
    \]
    In other words, $v = O_k(r^{\alpha_Q+})$ for all $k\geq 2$.

    Using form \eqref{eqn:mss-rough} of the minimal surface system, $v$ solves
    \[
    \Delta v = F(Dv,D^2v) \equiv Dv\star Dv\star D^2v = O_1(r^{3\alpha_Q-4+}).
    \]
    From Lemma \ref{lem:ps}, on $\{r>r_0\}$ for $r_0$ large there is a solution $v_p$ to
    \[
    \Delta v_p = F(Dv,D^2v) \quad \text{which obeys} \quad v_p = O_{2,1-}(r^{3\alpha_Q-2+}).
    \]
    Now $v-v_p$ is a harmonic function which is $O(r^{\alpha_Q+})$, and hence by the discussion below \eqref{eqn:q-harmonic-L2} $v-v_p = O_{2,1-}(r^{k_0/Q+})$ for $k_0$ being the maximal integer satisfying $k_0/Q\leq \alpha_Q$.  We deduce that $v = O(r^{\max\{k_0/Q,3\alpha_Q-2\}+})$, and since $3\alpha_Q - 2 < \alpha_Q$, by our choice of $\alpha_Q \in [0, 1)$ as the minimal exponent for which $v = O(r^{\alpha_Q+})$ we must have $\alpha_Q = k_0/Q$.  The expansion \eqref{eqn:refined-concl1} then follows from the Fourier expansion \eqref{eqn:q-harmonic-fourier} for $v - v_p$.
\end{proof}

\section{Multiplicity-two tangent plane}\label{sec:sheeting}

We prove Theorem \ref{thm:main} in the case when the tangent cone at infinity is a multiplicity-two plane.  We first prove an Allard-type sheeting theorem, then prove good asymptotics for the $2$-valued graphing function, and then apply our general calibration Theorem \ref{thm:calibrated} to deduce holomorphicity.

\begin{lemma}[Multiplicity $2$ sheeting]\label{lem:freq-sheeting}
    Given $\eps>0$, there exists $\delta = \delta(m,\eps)\in (0,1)$ so that the following holds. Let $P = \R^2 \times \{0^m\}$, and $T$ be an area-minimizing integral $2$-current in $C_2(0)\subset\R^{2+m}$ satisfying
    \begin{equation}\label{eqn:freq-sheeting-hyp1}
        (\del T)\res C_2 = 0, \quad E_{T,P}(0,2)\leq\delta^2,
    \end{equation}
    and, for some $N_0\in [\eps,1-\eps]$,
    \begin{equation}\label{eqn:freq-sheeting-hyp2}
        |\Theta_T(0,\rho)-2|\leq\delta, \quad |N_{T,P}(\rho)-N_0|\leq \delta \quad \text{for all }1/4\leq\rho\leq 2.
    \end{equation}
    Then $|N_0 - 1/2| \leq \eps$, and up to flipping the orientation of $T$, we can write
    \begin{equation}\label{eqn:freq-sheeting-concl1}
        T\res (A^2_{1,1/2}\times\R^m) = \llbracket \graph_{2\R^2}(v)\rrbracket,
    \end{equation}
    for some smooth, minimal, function $v:2\R^2\cap A_{1,1/2}\to \R^m$ admitting the approximation
    \begin{equation}\label{eqn:freq-sheeting-concl2}
        |E_{T,P}(0,1)^{-1/2}v - (a\sqrt{r}\cos(\theta/2) + b\sqrt{r}\sin(\theta/2))|_{C^2}\leq \eps,
    \end{equation}
    for $a, b \in \{0^2\} \times \R^m$ satisfying
    \begin{equation}\label{eqn:freq-sheeting-concl3}
    |a|^2 = |b|^2 = \tfrac{3}{4\pi}, \quad \<a, b\> = 0.
    \end{equation}
\end{lemma}

\begin{remark}
Condition \eqref{eqn:freq-sheeting-concl3} is saying that $E_{T, P}(0, 1)^{-1/2} v$ is close to being holomorphic, relative to a suitable orthogonal complex structure.
\end{remark}


\begin{remark}\label{rem:freq-sheeting}
With essentially a verbatim proof, a similar sheeting theorem holds for $T$ near a density $3$ plane.  Precisely, if instead of $|\Theta_T(0, \rho) - 2| \leq \delta$ one assumes $|\Theta_T(0, \rho) - 3| \leq \delta$, then either
\[
T \res (A^2_{1, 1/2} \times \R^m) = \llbracket \graph_{2\R^2}(v_2) \rrbracket + \llbracket \graph_{\R^2}(v_1) \rrbracket
\]
or
\[
T \res (A^2_{1, 1/2} \times \R^m) = \llbracket \graph_{3\R^2}(v_3) \rrbracket ,
\]
where, writing $E = E_{T, P}(0, 1)$, $E^{-1/2} v_1$ is $\eps$-close to $0$, $E^{-1/2} v_2$ is $\eps$-close to being $1/2$-homogeneous and harmonic, and $E^{-1/2} v_3$ is $\eps$-close to being $k/3$-homogeneous and harmonic for some $k \in \{ 1, 2\}$.

Another case we mention is when the plane has general density $Q \in \N$, but one has a priori the more precise frequency pinching $|N_0 - p/Q| \leq \eps$ for $p, Q$ coprime and $\eps(Q)$ sufficiently small. In this case, the approximating Dirichlet-minimizer necessarily splits into $Q$ multiplicity-one sheets, so the arguments of Lemma \ref{lem:freq-sheeting} imply
\[
T \res (A^2_{1, 1/2} \times \R^m) = \llbracket \graph_{Q\R^2}(v) \rrbracket
\]
where $E^{-1/2} v$ is $\eps$-close to being $p/Q$-homogeneous and harmonic.  For general frequencies $N_0\in [\eps,1-\eps]$, the same proof as Lemma \ref{lem:freq-sheeting} implies that $T\res (A^2_{1,1/2}\times\R^m)$ decomposes as $T=T_1+\cdots+T_N$ for some $N\geq 2$, where each $T_i$ lies within the $\eps E^{1/2}$-neighborhood of either the plane $P$ or some branch of a function modeled on $z^{p/q}$ for $p,q$ coprime satisfying $0<p<q\leq Q$.
\end{remark}

\begin{proof}
    We proceed by contradiction. Suppose for fixed $\eps>0$, there is a sequence $\delta_i\to 0$ and area-minimizing $2$-currents $T_i$ satisfying \eqref{eqn:freq-sheeting-hyp1} and \eqref{eqn:freq-sheeting-hyp2} with $T_i,\delta_i$ in place of $T,\delta$ respectively, but for which one of the conclusions \eqref{eqn:freq-sheeting-concl1} or \eqref{eqn:freq-sheeting-concl2} fails for all $i$. Set $E_i:= E_{T_i,P}(0,1)$, and note that by Corollary \ref{cor:H-control} we have
    \begin{equation}\label{eqn:sheeting-00}
    \frac{1}{c(m)} E_{T_i, P}(0, 2) \leq E_i \leq 2^4 E_{T_i, P}(0, 2) \quad \forall i \gg 1.
    \end{equation}. 

    After flipping the orientation of $T_i$ as necessary, our hypotheses imply that $T_i\to 2\llbracket \R^2\rrbracket$ in $C_2$ as both currents and varifolds (cf.~Lemma \ref{lem:mass-by-L2}). By Theorem \ref{thm:Lipschitz-approx} (cf.~\cite[Section 6.1]{KW23a}) and \eqref{eqn:sheeting-00}, we can find $\sigma_i\uparrow 1$, Lipschitz functions $u_i:B^n_{2\sigma_i}\to \cA_2(\R^m)$, and closed subsets $K_i\subset B_{2\sigma_i}^n$ such that
    \begin{equation}\label{eqn:sheeting-0}
        T_i\res (K_i\times \R^m) = \mathbf{v}(u_i) \res (K_i\times\R^m);
    \end{equation}
    \begin{equation}\label{eqn:sheeting-1}
        |B^2_{2\sigma}\setminus K_i| + \|T_i\|((B^2_{2\sigma}\setminus K_i)\times\R^m) \leq c_\sigma E^{1+\gamma}_i;
    \end{equation}
    \begin{equation}\label{eqn:sheeting-2}
        \sup_{B^2_{2\sigma}}|u_i|\leq c_\sigma E_i^{1/2}, \quad \sup_{B^2_{2\sigma}}|Du_i|\leq c_\sigma E_i^{\gamma}
    \end{equation}
    for every $\sigma\in (0,\sigma_i]$, where $c_\sigma = c_\sigma(m,\sigma)$ and $\gamma = \gamma(m)$. We note that the first inequality in \eqref{eqn:sheeting-2} comes from Lemma \ref{lem:Linfty-by-L2}. We also have for any $\sigma\in (0,\sigma_i]$,
    \[
    \int_{B^2_{2\sigma}}|Du_i|^2 dx \leq \int_{C_{2\sigma}}|\vec{T}_i-\vec{P}|^2\, d\|T_i\| + c(m,\sigma)E_i^{1+\gamma} \leq c(m,\sigma)E_i.
    \]
    Therefore, by \cite[Theorem 4.2]{DeLellis_Spadaro_2014} (cf.~\cite[Theorem 2.19]{Almgren_2000}) there is $w\in W^{1,2}(B^2_2;\cA_2(\R^m))$ which is locally Dirichlet-minimizing on $B_2^2$ for which $E_i^{-1/2}u_i\to w$ pointwise a.e.~on $B^2_2$, strongly in $L^2_{\text{loc}}(B^2_2)$, and also $E_i^{-1/2}|Du_i|\to |Dw|$ in $L^2_{\text{loc}}(B_2^2)$. Arguing as in \cite[Lemma 6.5]{KW23a}, we also have the convergence $N_{T_i,P}(\rho)\to N_w(\rho)$ for every $\rho\in (1/4,2)$.

    From \eqref{eqn:sheeting-1}, the choice of normalization, and the strong $L^2(B^n_1)$ convergence $E_i^{-1/2}u_i\to w$, we have $\int_{B_1}|w|^2 = 1$ and so $w\not\equiv 0$. Our hypotheses also give $N_w(\rho) \equiv N_0\in (0,1)$ for all $\rho\in (1/4,2)$. Therefore, $w|_{A_{2,1/4}}$ is $N_0$-homogeneous, and so as $w$ is Dirichlet-minimizing locally on $B_2^2$ and in particular has discrete singular set, by analyticity we know that $w$ must be $N_0$-homogeneous on all of $B_2$. Furthermore, the average $w_a$ of $w$ is a single-valued, $N_0$-homogeneous harmonic function, but since $N_0\in (0,1)$ we must have $w_a\equiv 0$.

    Since the singular set of $w$ is discrete and $w$ is homogeneous, we must have $\sing(w)\subseteq\{0\}$ -- in fact, we have $\{x:w(x)=2\llbracket 0\rrbracket\}\subseteq\{0\}$, as this set must also be discrete, and so if this were not the case we would need $w\equiv 2\llbracket w_a\rrbracket \equiv 2\llbracket 0\rrbracket$, contradicting $w\not\equiv 0$. Therefore, for every $x\in B_2^2\setminus \{0\}$, there is a radius $\rho = \rho(x)>0$ and harmonic functions $w_1,w_2:B_\rho^2(x)\to \R^m$ so that $w = \llbracket w_1(y)\rrbracket + \llbracket w_2(y)\rrbracket$ on $B_\rho^2(x)$, with $w_1\equiv -w_2$ and $w_1(y)\neq 0$ for all $y\in B_\rho^2(x)$.

    We argue that $N_0 = 1/2$.  The graph of $w|_{S^1}$ must be a connected double-cover of $S^1$ (otherwise, $w$ would split globally into two single-valued, $N_0$-homogeneous harmonic functions, which would necessarily be zero as $N_0\in (0,1)$, and thus would contradict $w\not\equiv 0$). So, we can lift $w|_{S^1}$ to a single-valued function $\hat{w}$ defined on $2S^1$, which extends to $2\R^2$ as an $N_0$-homogeneous harmonic function satisfying, in polar coordinates $x=r\theta$,
    \[
    w(r,\theta) = \llbracket \hat{w}(r,\theta)\rrbracket + \llbracket \hat{w}(r,\theta+2\pi).\rrbracket
    \]
    From the Fourier expansion \eqref{eqn:q-harmonic-fourier} of $\hat{w}$, we discover that the only possible value for $N_0\in (0,1)$ is $N_0 = 1/2$, corresponding to
    \[
    \hat{w}(r,\theta) = A \sqrt{r} \sin(\theta/2) + B \sqrt{r} \cos(\theta/2).
    \]
    Since $w$ is a $2$-valued $1/2$-homogeneous Dirichlet-minimizing function, it follows from \cite[Theorem 3.2]{micallef-white} that $|A|^2 = |B|^2$ and $\< A, B\> = 0$, and since $||w||_{L^2(B_1)} = 1$ we get the normalization \eqref{eqn:freq-sheeting-concl3}.

    We now argue that for every $z\in A^2_{3/2,1/4}$ and $\eps>0$, there is a radius $\rho = \rho(w,z,\eps)>0$ so that, writing $w|_{B_{2\rho}(z)} = \llbracket w_1\rrbracket + \llbracket w_2\rrbracket$, we can decompose for $i\gg1$,
    \begin{equation}\label{eqn:sheeting-a1}
        T_i\res C_\rho(z) = T_{i,1} + T_{i,2},\quad \text{with}\quad \|T_i\res C_{\rho}(z)\| = \|T_{i,1}\| + \|T_{i,2}\|,
    \end{equation}
    where $T_{i,j}$ are non-zero, integral $2$-currents with zero boundary in $C_\rho(z)$, which satisfy
    \begin{equation}\label{eqn:sheeting-a2}
        \sup_{y\in C_\rho(z)\cap \spt\|T_{i,j}\|}d(y,\graph_{\R^2}(E_i^{1/2}w_j))\leq \eps E_i^{1/2} \quad \text{for }j=1,2.
    \end{equation}
    From \eqref{eqn:sheeting-a1} and our original hypothesis that $T_i\to 2\llbracket \R^n\rrbracket$ in $C_2$ as both currents and varifolds, it follows that each $T_{i,j}$ is area-minimizing and $T_{i,j}\to \llbracket \R^2\rrbracket$ in $C_\rho(z)$ as $i\to\infty$, and so from Allard's regularity theorem $T_{i,j}\res C_{\rho/2}(z)$ is a smooth multiplicity-one graph over $\R^2$. Using \eqref{eqn:sheeting-a2} and a covering argument then implies that $T_i \res (A^2_{1,1/2}\times\R^m)$ is a smooth, multiplicity-one, graph of a function $v$ satisfying \eqref{eqn:freq-sheeting-concl2}, which is the desired conclusion of the lemma.

    We now prove \eqref{eqn:sheeting-a1} and \eqref{eqn:sheeting-a2}. Set $h:= |w_1(z)-w_2(z)| > 0$. For $\beta = \beta(m,\eps)>0$ to be fixed below, ensure that $\rho = \rho(w,z,\beta)>0$ is chosen sufficiently small so that
    \begin{equation}\label{eqn:sheeting-a3}
        \max\{\osc_{B_{2\rho}(z)}w_1,\osc_{B_{2\rho}(z)}w_2\}\leq\beta h.
    \end{equation}
    Now write
    \[
    F_i := \int_{B^2_{7/4}}\cG(E_i^{-1/2}u_i,w)^2\, dx,
    \]
    and note that by assumption we have $F_i\to 0$. From \eqref{eqn:sheeting-a3} and the definition of $F_i$ we have
    \begin{align*}
        \int_{C_{7/4}}d(y,\graph_{2\R^2}(E_i^{1/2}w))^2\, d\|T_i\| & \leq c\int_{K_i\times\R^m}\cG(u_i,E_i^{1/2}w)^2\, dx + cE_i^{1+\gamma}\\
        & \leq c(m)E_i(F_i+E_i^\gamma).
    \end{align*}
    Define then the union of planes
    \[
    \bm{\pi}_i := (E_i^{1/2}w_1(z) + \R^2)\cup (E_i^{1/2}w_2(z) + \R^2).
    \]
    Then we get
    \begin{align*}
        \rho^{-n-2}\int_{C_{2\rho}(z)}d(y,\bm{\pi}_i)^2\, d\|T_i\| & \leq 2\rho^{-n-2}\int_{C_{2\rho}(z)}d(y,\graph_{2\R^2}(E_i^{1/2}w))^2\, d\|T_i\|\\
        & \hspace{15em} + c\rho^{-2}E_i\beta^2 h^2\\
        & \leq c(m)\rho^{-2}E_i(\rho^{-n}F_i+\rho^{-n}E_i^\gamma+\beta^2h^2).
    \end{align*}
    By the improved $L^2-L^\infty$ estimate Theorem \ref{thm:improved-Linfty-L2} (applied at scale $\rho$) we have that $\spt\|T_i\|\cap C_\rho(z)$ lies in the $R_i$-neighborhood of $\bm{\pi}_i$, where
    \[
    R_i:= c_1(m)E_i^{1/2}(\rho^{-n/2}F_i^{1/2}+\rho^{-n/2}E_i^{\gamma/2}+\beta h).\]
    However, noting that the planes in $\bm{\pi}_i$ are parallel and at distance $E_i^{1/2}h$ from each other, we can ensure $\beta(m)$ is sufficiently small and $i\gg1$ (depending on $\rho$) to get
    \[
    R_i<\frac{1}{100}E_i^{1/2}h.
    \]
    Therefore, if we define
    \[
    T_{i,j}:= T_i \res [C_\rho(z)\cap B_{2R_i}(E_i^{1/2}w_j(z)+\R^2)],
    \]
    then $(\del T_{i,j})\res C_\rho(z) = 0$ and the decomposition \eqref{eqn:sheeting-a1} holds. The only thing left to verify is that $T_{i,j}\neq 0$. Indeed, if not, then we would have (without loss of generality) $T_i = T_{i,1}$ lies in $C_\rho(z)$ within a $\frac{1}{100} h E_i^{1/2}$-neighborhood $E_i^{1/2}(w_1(z) + \R^2)$ for infinitely many $i$, giving the contradiction
    \[
    F_i \geq \int_{B^2_\rho(z) \cap K_i} \cG(E_i^{-1/2} u_i, w)^2 dx \geq \frac{1}{2} h^2 |B_\rho^2(z) \cap K_i| \geq \frac{1}{4} h^2 \pi \rho^2, \quad i \gg 1.
    \]
    So, we must have $T_{i,j}\neq 0$ for each $i,j$, which proves \eqref{eqn:sheeting-a1}. The estimate \eqref{eqn:sheeting-a2} then follows by ensuring $\beta = \beta(m,\eps)$ is sufficiently small. We also note that one can take $\rho = \rho(w,\beta)$ in this argument, i.e.~$\rho$ is independent of the point $z\in A^2_{1,1/2}$, because one has a uniform lower bound on $w_1-w_2$, as well as an upper bound on $|Dw_i|$, on $A^2_{1,1/2}$.
\end{proof}

We now use Lemma \ref{lem:freq-sheeting} to prove an Allard-type sheeting theorem at infinity for area-minimizers with density $2$ at infinity.

\begin{theorem}[Regularity near infinity]\label{thm:mult-two-reg}
Let $T$ be a non-planar, area-minimizing $2$-current in $\R^{2+m}$ with $\del T = 0$ and $\Theta_T(\infty)=2$.  Suppose the tangent cone $T$ at infinity is $2 \llbracket R^2 \rrbracket$.  Then, after translating and scaling $T$ as necessary, we can find a smooth function $v : 2\R^2 \setminus \overline{B_1} \to \R^m$ so that
\begin{equation}\label{eqn:mult-two-reg-concl1}
T \res (\R^{2+m} \setminus \overline{C_1}) = \llbracket \graph_{2\R^2}(v) \setminus \overline{C_1} \rrbracket,
\end{equation}
and which admits the expansion
\begin{equation}\label{eqn:mult-two-reg-concl2}
v(r, \theta) = A_1 r^{1/2} \cos(\theta/2) + B_1 r^{1/2} \sin(\theta/2) + O_2(r^{-1/2+}),
\end{equation}
for $A_1, B_1 \in \R^m$.  Additionally, one can assume after a rigid motion in $\{0^2\}\times \R^{m}$ that
\begin{equation}\label{eqn:mult-two-reg-concl3}
A_1 = a e_3, \quad B_1 = a e_4, \quad a \in \R \text{ non-zero},
\end{equation}
so that identifying $\R^4 \times \{0^{m-2}\} \equiv \C^2$ the map $z \equiv r e^{i\theta} \mapsto (z^2, A_1 r \cos(\theta) + B_1 r \sin(\theta))$ is holomorphic.
\end{theorem}

\begin{remark}
We highlight that, unlike a general stationary solution as considered in Lemma \ref{lem:refined-asymptotics}, the expansion \eqref{eqn:mult-two-reg-concl2} has no $\log(r)$ term.
\end{remark}

\begin{proof}
    Write $\bC = 2\llbracket \R^2 \rrbracket$ for the tangent cone of $T$ at infinity; we know by Theorem \ref{thm:unique} that $\bC$ is unique and, there exists $\alpha = \alpha(m)\in (0,1)$ such that for any $\eta>0$ and $\alpha^\prime<\alpha$ we have
    \begin{equation}\label{eqn:structure-1}
        d_H(\spt\|T\|\cap B_R,\spt\|\bC\|\cap B_R)\leq R^{1-\alpha} \leq \eta R^{1-\alpha^\prime} \quad \text{for all }R\gg 1.
    \end{equation}

From Corollary \ref{cor:freq-mono} and \eqref{eqn:structure-1}, provided $\eta = \eta(m,\alpha^\prime)$ is sufficiently small (which we may assume), the quantity
    \[
    N_{T,P}(\rho)\exp(-c \eta^{2\gamma}\rho^{-2\alpha^\prime\gamma} + c^\prime (\eta^2 \rho^{-2\alpha^\prime})^{\gamma - \gamma_1}(\rho^{-2} D_{T,P}(\rho))^{\gamma_1})
    \]
    is non-decreasing for $\rho\gg 1$ (for suitable constants $c, c'$), and hence has a limit $M\in [0,\infty]$. Since this quantity is $>0$ (it is well-defined as $T$ is not planar) we know $M>0$. Moreover, since by Lemma \ref{lem:tilt-by-L2} and \eqref{eqn:structure-1} we have
    \[
    \rho^{-2}D_{T,P}(\rho) \leq \frac{1}{2}\rho^{-n}\int_{C_\rho}\|\pi_T-\pi_P\|^2\, d\|T\| \leq cE_{T,P}(0,2\rho) \to 0
    \]
    and get that $M = \lim_{\rho\to\infty} N_{T,P}(\rho)$ also.

    We now claim that $M\leq1-\alpha$. To see this, we argue by contradiction: suppose that $M>1-\alpha$. If $M<\infty$, set $M^\prime := (M+1-\alpha)/2>1-\alpha$, otherwise if $M=\infty$, then set $M^\prime:=1$. Then, choosing $\eps = \eps(m,M,M^\prime,\alpha^\prime)$ sufficiently small and $\tau = \tau(\eps,T)$ sufficiently large, from Corollary \ref{cor:H-control} we have, if $M<\infty$
    \begin{equation}\label{eqn:structure-2}
    H_{T,P}(\rho) \geq (\rho/\tau)^{2(M-\eps)-c\eps^{2\gamma}}H_{T,P}(\tau) \geq c(\tau,\eps,M,T)\rho^{2M^\prime} \quad \text{for all }\rho>\tau.
    \end{equation}
    for some constant $c(\tau,\eps,M,T)>0$ (of course we know that $H_{T,P}(\tau)>0$ as $T$ is assumed to be non-planar); if instead $M=\infty$, the same argument works just replacing $M-\eps$ in the first inequality by $2$, say. However, on the other hand, from \eqref{eqn:structure-1} we have, for all $\rho\gg 1$,
    \begin{align*}
        H_{T,P}(\rho) & \leq 2\rho^{1-2}\int_{C_\rho\setminus C_{\rho/2}}\left(\sup_{x\in\spt\|T\|\cap C_{2\rho}}d_P(x)^2\right)|\nabla r|^2\frac{1}{r}\, d\|T\|(x)\\
        & \leq c\rho^{-1}\cdot \rho^{2-2\alpha}\cdot\frac{1}{\rho}\cdot\|T\|(C_\rho)\\
        & \leq c\rho^{2-2\alpha}.
    \end{align*}
    Combining this inequality with \eqref{eqn:structure-2} gives $c\rho^{2M^\prime}\leq \rho^{2-2\alpha}$ for all $\rho\gg 1$, which is a contradiction since $M^\prime>1-\alpha$. This shows that we must have $M\leq 1-\alpha$.

    We deduce there is an $\eps'>0$ for which $M\in [\eps',1-\eps']$. Since $N_{T,P}(\rho)\to M$, $\Theta_T(0,\rho)\to 2$, and $E_{T,P}(0,\rho)\to 0$ as $\rho\to\infty$, we can therefore apply Lemma \ref{lem:freq-sheeting} on large cylinders to deduce that $T\res (\R^{2+m}\setminus \overline{C_R}) = \llbracket \graph_{2\R^2}(v)\setminus \overline{C_R}\rrbracket$ for some large $R>1$ and some smooth $v:2\R^2\setminus \overline{B_1}\to \R^m$. The asymptotics $v = O_2(r^{1-\alpha+})$ then follow from \eqref{eqn:structure-1} and \eqref{eqn:freq-sheeting-concl2}.  From Lemma \ref{lem:refined-asymptotics}, after enlarging $R$ if necessary we deduce $v$ admits the expansion 
    \[
    v(r, \theta) = A_1 r^{1/2} \cos(\theta/2) + B_1 r^{1/2} \sin(\theta/2) + A_0 \log(r) + B_0 + O_2(r^{-1/2+})
    \]
    for $r \geq R$.
    
    To verify the holomorphicity condition \eqref{eqn:mult-two-reg-concl3}, we observe that if $v_\rho(r, \theta) := \rho^{-1} v(\rho r, \theta)$, then
    \[
    E_{(1/\rho)_\# T, P}(0, 1)^{-1/2} v_\rho = \frac{A_1 \sqrt{r} \cos(\theta/2) + B_1 \sqrt{r} \sin(\theta/2)}{\sqrt{2\pi/3} \sqrt{|A_1|^2 + |B_1|^2}} + O(\rho^{-1/2+}).
    \]
    for $1/2 \leq r \leq 1$, and hence by applying Lemma \ref{lem:freq-sheeting} conclusions \eqref{eqn:freq-sheeting-concl2}, \eqref{eqn:freq-sheeting-concl3} to the rescaled current $(1/\rho)_\# T$ and taking $\rho \to \infty$ we deduce $|A_1| = |B_1|$ and $\< A_1, B_1\> = 0$.
        
    We can obviously assume $B_0 = 0$ by translating $T$.  We now show $A_0 = 0$.  To prove this, we suppose $A_0 \neq 0$ and derive a contradiction by constructing a competitor to $T$ with less area.  
We start by observing that
\begin{align}\label{eqn:mult-two-reg-1}
Jv &= 1 + \frac{1}{2} |Dv|^2 + \frac{1}{8} |Dv|^4 - \frac{1}{4} \sum_{i, j} \< D_i v, D_j v \>^2 + O(|Dv|^6) .
\end{align}

Define
\[
v_1 = A_1 r^{1/2} \cos(\theta/2) + B_1 r^{1/2} \sin(\theta/2), \quad v_0 = A_0 \log(r) + B_0, \quad v_{-1} = v - v_1 - v_0,
\]
so that $v = v_1 + v_0 + v_{-1}$, and then for $1 \ll \rho \ll R$, define $w(r, \theta)$ by
\[
w(r, \theta) = \left\{ \begin{array}{l l} v_1(r, \theta) \phi(r) + v_0(R) + v_{-1}(r, \theta) \phi(r)  & r \geq \rho \\ v_0(R) & r \leq \rho \end{array} \right.
\]
where $\phi(t)$ is a monotone function which is $\equiv 0$ for $t \leq \rho$ and $\equiv 1$ for $t \geq 2\rho$.  Define $T' = \llbracket \graph_{2\R^2}(w) \rrbracket$.  By construction, $w(R, \theta) = v(R, \theta)$, and since $w$ extends as a single-valued graph across zero, $T'$ is a well-defined integral $2$-current without boundary in $\R^{2+m}$.

Noting that
\[
\int_{2S^1} \< Dv_1(r, \theta), Dv_0(r, \theta) \> d\theta = 0 \quad \forall r \geq \rho,
\]
we compute
\begin{align*}
\bM(T \res C_R)
&\geq \int_{2\R^2 \cap A_{R, 2\rho}} Jv dx \\
&= \int_{2 \R^2 \cap A_{R, 2\rho}} 1 + \frac{1}{2} |Dv_1|^2 + \< Dv_1, Dv_{-1}\> + \frac{1}{2} |Dv_0|^2  \\
&\quad\quad\quad\quad  + \frac{1}{8} |Dv_1|^4 - \frac{1}{4} \sum_{i, j} \< D_i v_1 , D_j v_1 \>^2 + O(r^{-5/2+}) dx \\
&= \int_{2 \R^2 \cap A_{R, 2\rho}} Jw + \frac{1}{2}|Dv_0|^2 + O(r^{-5/2+}) dx \\
&\geq \int_{2 \R^2 \cap A_{R, 2\rho}} Jw dx + 2\pi |A_0|^2 \int_{2\rho}^R \frac{1}{r} dr - c(T, \rho)\\
&\geq \bM(T' \res C_R) + 2\pi |A_0|^2 \log(R/2\rho) - c(T, \rho) \\
&> \bM(T' \res C_R)
\end{align*}
provided $R(\rho, T)$ is chosen sufficiently large.  Since $\del (T' \res C_R) = \del (T \res C_R)$, this contradicts area-minimality of $T$, and therefore we must have $A_0 = 0$.  This proves Theorem \ref{thm:mult-two-reg}.
\end{proof}

\begin{proof}[Proof of Theorem \ref{thm:main} when $\bC$ is a multiplicity-two plane]
After a rotation we can assume the tangent cone of $T$ at infinity is $2\llbracket \R^2 \rrbracket$.  We first observe that since $\Theta_T(\infty) = 2$, the monotonicity formula and Allard's regularity theorem imply either $T = \llbracket p_1 + \R^2 \rrbracket + \llbracket p_2 + \R^2 \rrbracket$ for some $p_1, p_2 \in \R^{2+m}$, or $T$ is entirely regular, multiplicity-one, and connected.  We henceforth assume the latter case.

Apply Theorem \ref{thm:mult-two-reg} to find a $v : 2 \R^2 \setminus \overline{B_1} \to \R^m$ admitting the expansion
\begin{align}
v(r, \theta) &= a e_3 r^{1/2} cos(\theta/2) + a e_4 r^{1/2} \sin(\theta/2) + O_2(r^{-1/2+}) \label{eqn:main-1} \\
&=: v_0(r, \theta) + O_2(r^{-1/2+}) \nonumber
\end{align}
for some non-zero $a \in \R$, so that after scaling $T$ and applying a further rigid motion in $\{0^2\} \times \R^m$, we have
\[
T \res (\R^{2+m} \setminus \overline{C_1}) = \llbracket \graph_{2\R^2}(v) \setminus \overline{C_1} \rrbracket.
\]
By \eqref{eqn:main-1}, if $e \in (\R^4)^\perp$ then the map $x \mapsto \< x, e \>$ is bounded and tending to zero as $|x| \to \infty$, so by Lemma \ref{lem:vector-rigidity} we get $\< x, e \> \equiv 0$ on $\spt\, T$.  Therefore $\spt\, T \subset \R^4 \times \{0^{m-2}\}\equiv \C^2$.

Define $T' = \llbracket \{ (z, w) \in \C^2 : w^2 = a^2 z \} \rrbracket$, and then observe for $R$ suitable large we have
\begin{equation}\label{eqn:main-2}
T' \res (\R^{2+m} \setminus \overline{C_R}) = \llbracket \graph_{2\R^2}(v_0) \setminus \overline{C_R} \rrbracket
\end{equation}
From \eqref{eqn:main-1}, \eqref{eqn:main-2} and the implicit function theorem, enlarging $R$ as necssary, we can find a $u : T' \setminus \overline{C_R} \to T'^\perp$ so that
\begin{equation}\label{eqn:main-3}
T \res (\R^{2+m} \setminus \overline{C_R}) = \llbracket \graph_{T'}(u) \setminus \overline{C_R} \rrbracket 
\end{equation}
and $|u(x)| \leq |x|^{-1/4}$.

In light of \eqref{eqn:main-2}, \eqref{eqn:main-3}, and the holomorphicity of $T'$, we can apply Theorem \ref{thm:calibrated} to deduce $T$ is holomorphic in $\C^2$ also.  It follows by Stoll's theorem \cite{stolzenberg}, or the same argument as in \cite{Edelen_Franco_Minter}, that $\spt\, T$ is the zero set of a degree-2 holomorphic polynomial on $\C^2$.  This proves Theorem \ref{thm:main}.

Lastly, we justify the claim in the introduction that if $T^2\subset\R^n$ is a complete area-minimizer asymptotic to $Q\llbracket P\rrbracket$ (so in particular $\Theta_T(\infty)=Q$) and $N_{T,P}(\infty) = 1/Q$, then $T$ must be algebraic.  Indeed, in this situation one can prove that $T$ is regular at infinity by Remark \ref{rem:freq-sheeting}, at which point the condition $N_T(\infty)=1/Q$ forces the asymptotic expansion from Lemma \ref{lem:refined-asymptotics} to become
\[
v(r,\theta) = A_1 r^{1/Q}\cos(\theta/Q)+ B_1 r^{1/Q}\sin(\theta/Q) + A_0 \log(r) + B_0 + O(r^{-1/Q+}).
\]
The same proof as in Theorem \ref{thm:mult-two-reg} implies $|A_1| = |B_1|$, $\< A_1, B_1\> = 0$ (since the $1/Q$-homogeneous part is a $Q$-valued Dirichlet-minimizing map), and $A_0 = 0$ (by the minimization property).  We can repeat the above proof in this case to deduce $T$ is contained in, and holomorphic in, an affine copy of $\C^2$.
\end{proof}

\section{Multiplicity-one tangent cone}

We handle the case when the area-minimizing $2$-current $T$ has a multiplicity-one tangent cone at infinity.

\begin{theorem}[Regularity near infinity]\label{thm:mult-one-reg}
Let $T$ be an area-minimizing $2$-current in $\R^{2+m}$ with $\del T = 0$ and $\Theta := \Theta_T(\infty) < \infty$.  Suppose the tangent cone at infinity $\bC = \sum_{i=1}^{\Theta} \llbracket P_i \rrbracket $ of $T$ has multiplicity-one.  Then, after scaling $T$ as necessary, we can find smooth functions $u_i : P_i \to P_i^\perp$ so that
\begin{equation}\label{eqn:mult-one-reg-concl1}
T \res (\R^{2+m} \setminus \overline{B_1}) = \sum_{i=1}^\Theta \llbracket \graph_{P_i}(u_i) \setminus \overline{B_1}\rrbracket,
\end{equation}
and which admit the expansions
\begin{equation}\label{eqn:mult-one-reg-concl2}
u_i(x = r\theta) = B_{i0} + O_2(r^{-1}) , \quad B_{i0} \in P_i^\perp.
\end{equation}
\end{theorem}

\begin{proof}
The existence of smooth functions $u_i : P_i \to P_i^\perp$ satisfying $u_i = O(r^{1-\alpha})$ for which \eqref{eqn:mult-one-reg-concl2} holds follows from Theorem \ref{thm:unique} and Allard's theorem.  Lemma \ref{lem:refined-asymptotics} then implies each $u_i$ has the asymptotics
\[
u_i(r\theta) = A_{i0} \log(r) + B_{i0} + O_2(r^{-1}).
\]
We just need to show that each $A_{i0} = 0$.

The proof is similar to Theorem \ref{thm:mult-two-reg}, where we shall suppose some $A_{i0} \neq 0$ and construct a competitor to $T$ with less mass.  
Let $u_{i0} = A_{i0} \log(r) + B_{i0}$.  For $1 \ll \rho \ll R$, define
\[
w_i(r\theta) = u_{i0}(R, \theta) + \phi(r) (u_i(r, \theta) - u_{i0}(r, \theta))
\]
for $\phi(t)$ a smooth monotone function which is $\equiv 0$ for $t \leq \rho$ and $\equiv 1$ for $t \geq 2\rho$.  Define
\[
T' = \sum_{i=1}^\Theta \llbracket \graph_{P_i}(w_i) \rrbracket,
\]
and then we have $\del (T' \res \cap_{i=1}^\Theta C_R(0, P_i)) = \del (T \res \cap_{i=1}^\Theta C_R(0, P_i))$.

But now, using the expansion \eqref{eqn:mult-two-reg-1}, we compute
\begin{align*}
\bM(T \res \cap_{i=1}^\Theta C_R(0, P_i))
&\geq \sum_{i=1}^\Theta \int_{P_i \cap A_{R, 2\rho}} Ju_i dx \\
&= \sum_{i=1}^\Theta \int_{P_i \cap A_{R, 2\rho}} 1 + \frac{1}{2} |Du_{i0}|^2 + O(r^{-3}) dx \\
&= \sum_{i=1}^\Theta \int_{P_i \cap A_{R, 2\rho}} Jw_i + \frac{1}{2} |Du_{i0}|^2 + O(r^{-3}) dx \\
&\geq \sum_{i=1}^\Theta \int_{P_i \cap A_{R, 2\rho}} Jw_i dx + \sum_{i=1}^\Theta \pi |A_{i0}|^2 \int_{2\rho}^R \frac{dr}{r} - c(T, \rho) \\
&\geq \bM(T' \res \cap_{i=1}^\Theta C_R(0, P_i)) + \sum_i \pi |A_{i0}|^2 \log(R/2\rho) - c(T, \rho) \\
&> \bM(T' \res \cap_{i=1}^\Theta C_R(0, P_i)),
\end{align*}
provided $R(T, \rho)$ is chosen sufficiently large.  However this contradicts area-minimality of $T$, and so we must have $A_{i0} = 0$ for all $i$.
\end{proof}

\begin{proof}[Proof of Theorem \ref{thm:main-mult-one}, and Theorem \ref{thm:main} when $\bC$ has multiplicity-one]
Recall that $\bC = \sum_{i=1}^\Theta \llbracket P_i \rrbracket$ is the tangent cone at infinity for $T$, and by assumption the $P_i, P_j$ are pairwise-transverse.  It follows from \cite{Morgan_1982} that after a rigid motion in $\R^{2+m}$ the span $V_1 + \cdots + V_\Theta = \R^{2k} \times \{ 0^{2+m-2k} \}$ for some $k \geq 1$, and if we identify $\C^k \equiv \R^{2k}$ then each $P_i$ is holomorphic in $\C^k$.  Write $\omega$ for the standard Kahler form on $\C^k$, so that $\omega$ restricted to each $P_i$ is the positive volume element.

Let $M_1, \ldots, M_l$ be the connected components of $\reg\, T$, and write $T_i = T \res M_i$.  Since $\sing\, T$ is discrete, we have
\[
T = T_1 + \cdots + T_l, \quad \|T\| = \|T_1\| + \cdots + \|T_l\|,
\]
and hence each $T_i$ is area-minimizing.  From the monotonicity formula $l \leq \Theta_{T_1}(\infty) + \cdots + \Theta_{T_l}(\infty) = \Theta < \infty$.  Let us write $\bC_i$ for the tangent cone of $T_i$ at infinity, so that
\begin{equation}\label{eqn:main-mult-one-1}
\bC = \bC_1 + \cdots + \bC_l, \quad \|\bC\| = \|\bC_1\| + \cdots + \|\bC_l\|.
\end{equation}
\eqref{eqn:main-mult-one-1} implies that each $\bC_i$ is holomorphic and multiplicity-one in $\C^k$ also, and then combined with the constancy theorem we deduce $T_i = \llbracket M_i \rrbracket$ is multiplicity-one also.

Apply Theorem \ref{thm:mult-one-reg} to find smooth functions $u_{j} : P_j \to P_j^\perp$ satisfying the expansions \eqref{eqn:mult-one-reg-concl2}, so that after scaling $T$ as necessary \eqref{eqn:mult-one-reg-concl1} holds.  If $e \in \R^{2+m}$ is any vector orthogonal to $\R^{2k}$, then from \eqref{eqn:mult-one-reg-concl2} we have that $x \mapsto \< x, e \>$ is bounded on $\spt\, T$, and hence by Lemma \ref{lem:vector-rigidity} is constant on each $\reg\, T_i \equiv M_i$.  We can therefore find $p_1, \ldots, p_l \in (\R^{2k})^\perp$ so that $\spt\, T_i \subset  p_i + \R^{2k} $.

From \eqref{eqn:mult-one-reg-concl2} we can apply Theorem \ref{thm:calibrated} to each $T_i - p_i$ in $\C^k \equiv \R^{2k}$ with the calibrating form $\omega$ to deduce that $T_i - p_i$ is holomorphic in $\C^k$.  From Stoll's theorem \cite{stolzenberg} (or a similar argument to \cite{Edelen_Franco_Minter}), we can find an irreducible holomorphic polynomial $q_i(z_1, \ldots, z_k)$ of degree $= \Theta_{T_i}(\infty)$ on $\C^k$ so that
\[
T_i = \llbracket p_i + \{ q_i(z_1, \ldots, z_k) = 0 \} \rrbracket.
\]

In the case when $\Theta = 2$ and $P_1 \neq P_2$, then unless $T$ is already planar, the monotonicity formula implies that $T$ is irreducible, smooth, and multiplicity-one.  In this case we must have $k = 2$, and up to translation in $\R^{2+m}$ we get $T = \llbracket \{ p(z, w) = 0 \} \rrbracket$ for some holomorphic degree-2 polynomial $p(z, w)$ on $\C^2$.
\end{proof}

\section{Calibrations at infinity}

Here we prove the key mechanism behind our proof of holomorphicity: if an entire area-minimizing (or in some cases stationary) current $T$ decays sufficiently fast at infinity to a calibrated current $T'$, which itself is almost conical, then $T$ must also be calibrated and by the same form.


\begin{theorem}\label{thm:calibrated}
For $n \geq 2$, $m \geq 1$, let $T, T'$ be area-minimizing integral $n$-currents in $\R^{n+m}$ with $\del T = \del T' = 0$, and let $\omega$ be a closed $n$-form on $\R^{n+m}$ with comass $1$.  Suppose $T'$ is calibrated by $\omega$, is smooth and multiplicity-one outside $\overline{B_1}$, and satisfies
\begin{equation}\label{eqn:calibrated-hyp1}
\Theta_{T'}(\infty) < \infty, \quad \sup_{|x| > 1} |x| |A_{T'}(x)| < \infty,
\end{equation}
where $A_{T'}$ denotes the curvature of $\spt\, T'$.  Suppose $\Theta_T(\infty) < \infty$, and $T$ is asymptotic to $T'$ in the following sense
\begin{gather}
T \res (\R^{n+m} \setminus \overline{B_1}) = \llbracket \graph_{T'}(u) \setminus \overline{B_1} \rrbracket, \quad  |u(x)| \leq |x|^{-(n-2)/2-\alpha}, \label{eqn:calibrated-hyp2}
\end{gather}
for $u \in C^2(T' \setminus \overline{B_{1/2}}, T'^\perp)$, and $\alpha > 0$.

Then $T$ is calibrated by $\omega$ also.  Moreover, if $T' = \bC$ is conical, then it suffices to assume that $T$ is merely stationary and $u$ admits the weaker bound $|u(x)| \leq C |x|^{-(n-2)/2}$ for some constant $C$.
\end{theorem}

\begin{remark}
No regularity of $\omega$ is required in Theorem \ref{thm:calibrated}.  It would suffice to assume that $\omega$ is a flat $n$-cochain (as in \cite{federer:cochains}) which is distributionally closed and has comass $1$.
\end{remark}

\begin{proof}
Write $\nabla'$ for the connection derivative on $T'^\perp$, and $r = |x|$.  We first observe that by \eqref{eqn:calibrated-hyp1}, \eqref{eqn:calibrated-hyp2}, and standard estimates for minimal surfaces, we have the higher order estimates
\begin{gather}
r^{-1} |u| + |\nabla' u| + r |\nabla'^2 u| \leq c(T') r^{-\alpha - (n-2)/2 - 1} , \label{eqn:calibrated-1} 
\end{gather}
for all $r \gg 1$.  We second observe that \eqref{eqn:calibrated-hyp1} implies that whenever $w : T' \to T'^\perp$ is a $C^1$ function for which $|x|^{-1} |w(x)| + |\nabla' w(x)| \leq \eps(T')$ is sufficiently small, and $R > \rho \gg 1$ sufficiently large, then
\[
\haus^n(\graph_{T'}(w|_{A_{R, \rho}})) = \int_{A_{R, \rho}} J_{T'}(x, r^{-1} w(x), \nabla' w(x)) d\|T'\|(x)
\]
for $J_{T'}(x, p, q)$ a smooth function satisfying
\[
|J_{T'}(x, p, q) - 1| \leq c(T')(|p|^2 + |q|^2), \quad \forall |x| > R \gg 1.
\]

Take $R \gg 1$, and let $\phi(t)$ be a monotone decreasing function which is $\equiv 1$ if $t \leq R$, $\equiv 0$ if $t \geq 2R$, and $\phi' \leq 2/R$.  Define the currents
\begin{gather*}
S_R := \llbracket \graph_{T'}( \phi(|x|) u(x) |_{ A_{2R, R}})\rrbracket , \quad S'_R := \llbracket \graph_{T'}( (1-\phi(|x|)) u(x))|_{A_{2R, R}} \rrbracket, \\ 
T_R := T -  \llbracket \graph_{T'}(u|_{\R^{n+m}\setminus \overline{B_R}}) \rrbracket , \quad T'_R := T' \res B_R.
\end{gather*}
By construction we have
\begin{equation}\label{eqn:calibrated-3}
\del(T_R + S_R) = \del(T'_{2R}) \quad \text{and} \quad \del(T'_R + S'_R) = \del T_{2R}.
\end{equation}

We claim that
\begin{equation}\label{eqn:calibrated-4}
|\bM(T_R) - \bM(T'_R)| \leq c(T') R^{-2\alpha} \quad \forall R \gg 1.
\end{equation}
To see this, we first compute that
\begin{align*}
\bM(S_R) &= \int_{A_{2R, R}} J_{T'}(x, |x|^{-1} (\phi u), \nabla' (\phi u)) d\|T'\| \\
&\leq \bM(T' \res A_{2R, R}) + c \int_{A_{2R, R}} |x|^{-2} |u|^2  + |\nabla' u|^2 d\|T'\| \\
&\leq \bM(T' \res A_{2R, R}) + c \int_{A_{2R, R}} |x|^{-2\alpha - (n-2) - 2} d\|T'\| \\
&\leq \bM(T' \res A_{2R, R}) + c(T') R^{-2\alpha}
\end{align*}
and similarly we have
\[
\bM(S'_R) \leq \bM(T'_{2R}) - \bM(T'_R) + c(\bC) R^{-2\alpha}.
\]
Use the area-minimizing property with \eqref{eqn:calibrated-3} to get
\begin{align*}
&\bM(T'_{2R}) \leq \bM(T_R) + \bM(S_R) \leq \bM(T_R) + \bM(T'_{2R}) - \bM(T'_R) + c R^{-2\alpha} \\
&\implies \bM(T'_R) \leq \bM(T_R) + c(\bC) R^{-2\alpha}
\end{align*}
and conversely
\begin{align*}
\bM(T_{2R}) \leq \bM(T'_R) + \bM(S'_R) &\leq \bM(T'_R) + \bM(T'_{2R}) - \bM(T'_R) + c R^{-2\alpha} \\
&\leq \bM(T'_{2R}) + 2 c(\bC) (2R)^{-2\alpha}.
\end{align*}
This proves our claim \eqref{eqn:calibrated-4}.

Now we use \eqref{eqn:calibrated-3}, \eqref{eqn:calibrated-4}, and the property that $T'$ is calibrated by $\omega$ to get
\begin{align*}
\bM(T_R) \geq T_R(\omega) &= T'_{2R}(\omega) - S_R(\omega) \\
&\geq \bM(T'_{2R}) - \bM(S_R) \\
&\geq \bM(T'_R) - c R^{-2\alpha} \\
&\geq \bM(T_R) - c(T') R^{-2\alpha},
\end{align*}
which implies
\[
\int | 1 - \omega[T_x T] |\, d\|T_R\| \equiv \bM(T_R) - T_R(\omega) \leq c(T') R^{-2\alpha} \to 0
\]
as $R \to \infty$, giving that $\omega[T_x T] \equiv 1$ on $T$.  Therefore $T$ is calibrated by $\omega$.

\vspace{3mm}

We now handle the case when $T' = \bC$.  Write $\nabla$ for the connection derivative on $\bC^\perp$.  It will be more natural to use spherical graphs rather than normal graphs over $\bC$.  For $v : \bC \cap A_{R, \rho} \to \bC^\perp$ define
\[
G_{R, \rho} (v) := \left \{ \frac{ x + v(x)}{|x + v(x)|} |x| : x \in \spt\, \bC, \rho < |x| < R \right \}.
\]
By \eqref{eqn:calibrated-hyp2} and the inverse function theorem, provided suitable $R(T) \gg 1$, we can rewrite
\[
T \res (\R^{n+m} \setminus \overline{B_R})  = \llbracket G_{\infty, R}(v) \rrbracket
\]
where $v : \bC \setminus \overline{B_R} \to \bC^\perp$ is a smooth function satisfying
\begin{equation}\label{eqn:calibrated-9}
r^{-1} |v| + |\nabla v| + r |\nabla^2 v| \leq C' r^{-\alpha - (n-2)/2-1}
\end{equation}
for $\alpha \geq 0$ and some $C' < \infty$.

By direct computation we can expand the area of a spherical graph as
\begin{align}
&\haus^n(G_{R, \rho}(v)) \nonumber \\
&= \int_\rho^R \int_\Sigma \Big( 1 + \frac{1}{2} \Big[ r^2 |\del_r(v/r)|^2 \nonumber \\
&\quad \quad \quad + r^{-2} \Big( \sum_i |(\del_{\theta^i} v)^\perp|^2 - (n-1)|v|^2 - \sum_{i, j} \< v ,  B^\Sigma_{ij}\>^2 \Big) \Big]  \nonumber \\
&\quad \quad \quad + h.o.t. \Big) r^{n-1} d\theta dr \nonumber \\
&=: \int_\rho^R \int_\Sigma \left( 1+ \frac{1}{2} \left[ r^2 |\del_r(v/r)|^2 + q_\Sigma(v/r, v/r) \right] + h.o.t.\right) r^{n-1}  d\theta dr, \label{eqn:calibrated-9.5}
\end{align}
where each term in $h.o.t.$ is at least cubic in $v/r, r\del_r(v/r), \del_{\theta^i}(v/r)$, and $q_\Sigma$ is the quadratic form associated to the Jacobi operator $L_\Sigma$ for $\Sigma \subset S^{n+m-1}$, defined by
\[
L_\Sigma v = \Delta^\perp_\Sigma v + \sum_{i, j} B^\Sigma_{ij} \< B^\Sigma_{ij} , v \> + (n-1) v .
\]
The operator $L_\Sigma$ is elliptic, self-adjoint, and so admits an $L^2(\Sigma, \Sigma^\perp)$-orthonormal basis of eigenfunctions $\{ \phi_i \}_i$ with corresponding eigenvalues $\{ \mu_i \}_i$, so that
\[
\int_\Sigma \phi_i L_\Sigma \phi_j d\theta = - \int_\Sigma q_\Sigma(\phi_i, \phi_j) d\theta = - \mu_i \delta_{ij} .
\]

The surface $G_{R, \rho}(v)$ is stationary if and only if $v$ satisfies an equation of the form
\[
\cL v := r^{-n} \del_r(r^{n+1} \del_r(v/r)) + r^{-2} L_\Sigma v = r^{-1} Q_\Sigma(x, v/r, \nabla v, r \nabla^2 v)
\]
where $Q_\Sigma(x, z, p, q)$ is a smooth function which is quadratic in $z, p, q$.  It follows by \cite{simon-extrema} (or by essentially a verbatim argument as Lemma \ref{lem:ps}) that if $f : \bC \setminus B_1 \to \bC^\perp$ is a smooth function satisfying $f = O(r^{-1-\gamma+})$ for some $\gamma > 0$, then there is a smooth $u : \bC \setminus \overline{B_1} \to \bC^\perp$ solving $\cL u = f$ and admitting the bound $u = O(r^{1-\gamma+})$.  By repeating the argument in Lemma \ref{lem:refined-asymptotics}, we can expand $v$ as\footnote{We remark that, for the purposes of Theorem \ref{thm:main-mult-one}, the expansion \eqref{eqn:mult-one-reg-concl2} for $u$ implies \eqref{eqn:calibrated-9.6} directly by the inverse function theorem, in which case much of the previous analysis is not needed.}
\begin{equation}\label{eqn:calibrated-9.6}
v(x = r\theta) = A r^{-(n-2)/2} \phi(\theta) + O(r^{-(n-2)/2 - \beta})
\end{equation}
for some (small) $\beta > 0$, and where either $A = 0$ (such as when $\alpha > 0$ or $-(n/2)^2$ is \emph{not} an eigenvalue of $L_\Sigma$), or $A \in \R$ and $\phi$ is an eigenfunction of $q_\Sigma$ with eigenvalue $-(n/2)^2$ and (for convenience) $\|\phi\|_{L^2(\Sigma)} = 1$.

Now let us define
\[
w(r\theta) = (2-\log r /\log R) v(r\theta), \quad R \leq r \leq R^2,
\]
and then define the currents
\[
S_R = \llbracket G_{R^2, R}(w) \rrbracket, \quad T_R = T \res B_R, \quad \bC_R = \bC \res B_R.
\]
By construction we have
\begin{equation}\label{eqn:calibrated-10}
\del (T_R + S_R) = \del \bC_{R^2}.
\end{equation}
(Note the definition of $T_R$ here is different from the first case we considered!)

Using that $\int_\Sigma q_\Sigma(\phi, \phi) d\theta = -(n/2)^2$, we can estimate
\begin{align}
&\bM(S_R) - \bM(\bC \res A_{R^2, R}) \nonumber  \\
&= \frac{1}{2} \int_R^{R^2} \int_\Sigma (r^2 |\del_r(w/r)|^2 + q_\Sigma(w/r, w/r) + O(r^{-3n/2}) ) r^{n-1} d\theta dr  \nonumber \\
&= \frac{1}{2} \int_R^{R^2} \int_\Sigma \Big( r^{-n} A^2 |\phi|^2 \Big|1/\log R + (n/2) (2-\log r/\log R)\Big|^2 + O(r^{-n-\beta}) \nonumber \\
&\quad \quad -(n/2)^2 A^2 |\phi|^2 (2-\log r / \log R)^2 r^{-n} + O(r^{-n-\beta}) \Big) r^{n-1} d\theta dr \nonumber  \\
&= \frac{A^2}{2} \int_R^{R^2} \left( \Big |1/\log R + (n/2) (2-\log r/\log R)\Big|^2 - (n/2)^2 (2-\log r/\log R)^2 \right) \frac{dr}{r} \nonumber \\
&\quad \quad + \int_R^{R^2} O(r^{-1-\beta}) dr \nonumber  \\
&= \frac{A^2}{2} \int_R^{R^2} \left( (\log R)^{-2} + n (\log R)^{-1} (2- \log r/\log R) \right) \frac{dr}{r} + O(R^{-\beta}) \nonumber \\
&= A^2 (n/4) + O(A^2 (\log R)^{-1}) + O(R^{-\beta}). \label{eqn:calibrated-11}
\end{align}

By the monotonicity formula for stationary minimal surfaces we trivially have $\bM(T_R) \leq \bM(\bC_R)$, but in the case when $A \neq 0$ we need a sharper estimate.  For $y = \frac{x + v(x)}{|x + v(x)|} |x| \in G(v)$, and $r = |x|$, we have the Hardt--Simon-type identity (\cite{hardt-simon})
\[
|r \del_r(v(x)/r)| = |\pi_{T_y G}^\perp(y/r)|(1 + O(|v(x)|/r) + O(|\nabla v(x)|)) ,
\]
which we can use with the remainder formula for the monotonicity formula and \eqref{eqn:calibrated-9.5} to get
\begin{align}
\bM(\bC_R) - \bM(T_R) &= R^n \int_{A_{\infty, R}} \frac{ |\pi_{T_y T}^\perp (y)|^2}{|y|^{n+2}} d\|T\|(y) \nonumber \\
&= R^n \int_{A_{\infty, R}} r^{-n-2} (r^4 |\del_r(v/r)|^2 + O(r^{-3n/2+2})) d\|\bC\| \nonumber \\
&= R^n \int_{A_{\infty, R}} r^{2-n} (n/2)^2 r^{-n-2} A^2 |\phi|^2 + O(r^{-2n-\beta}) d\|\bC\| \nonumber \\
&= R^n A^2 (n/2)^2 \int_R^\infty r^{-n-1} dr + R^n \int_R^\infty O(r^{-n-1-\beta}) dr \nonumber \\
&= A^2 (n/4) + O(R^{-\beta}). \label{eqn:calibrated-12}
\end{align}

We now combine \eqref{eqn:calibrated-10}, \eqref{eqn:calibrated-11}, \eqref{eqn:calibrated-12}, with the assumption that $\bC$ is calibrated by $\omega$, to compute
\begin{align*}
\bM(T_R)
\geq T_R(\omega) 
&= \bC_{R^2}(\omega) - S_R(\omega) \\\
&\geq \bM(\bC_{R^2}) - \bM(S_R) \\
&\geq \bM(\bC_R) - A^2 (n/4) + O( (\log R)^{-1}) \\
&\geq \bM(T_R) + O( (\log R)^{-1}).
\end{align*}
And then the proof proceeds as in the general case above.
\end{proof}

\appendix
\section{An $L^2$ estimate}

We prove the following $L^2$ estimate which was used in the proof of Theorem \ref{thm:unique}.

\begin{lemma}\label{lem:L2-by-mono}
    Fix $0<\sigma<\tau$. Let $V$ be a stationary $n$-varifold in $A_{\tau,\sigma/2}\subset\R^{n+m}$, and let $f:\R^{n+m}\setminus\{0\}\to \R$ be a non-negative, $2$-homogeneous $C^1$ function admitting the bound $\|f\|_{C^1(S^{n+m-1})}\leq M^2$. Set $\Lambda:= \sup_{\sigma\leq\rho\leq\tau}\rho^{-n}\|V\|(A_{\rho,\rho/2})$. Then we have the inequality
    \begin{align*}
        &\left|\tau^{-n-2}\int_{A_{\tau,\tau/2}} f(x)\, d\|V\| - \sigma^{-n-2}\int_{A_{\sigma,\sigma/2}}f(x)\, d\|V\|\right|\\
        & \hspace{7em} \leq \int_{A_{\tau,\tau/2}\cup A_{\sigma,\sigma/2}}\frac{|\pi^\perp_V(x)|^2}{|x|^{n+2}}\frac{f(x)}{|x|^2}\, d\|V\|\\
        & \hspace{12em}+ 4M^2\Lambda^{1/2}\log(\tau/\sigma)\left(\int_{A_{\tau,\sigma/2}}\frac{|\pi_V^\perp(x)|^2}{|x|^{n+2}}\, d\|V\|\right)^{1/2}.
    \end{align*}
\end{lemma}

\begin{proof}
    Take $\zeta:\R\to\R$ a non-negative, smooth function compactly supported in $[1/2,1]$. Fix also $\rho\in [\sigma,\tau]$. Consider the vector field
    \[
    X(x) := \zeta(|x|/\rho)f(x) x.
    \]
    Note that since $f(x)$ is $2$-homogeneous we have $\<x,D(f)\>=2f$. Let us abbreviate $x^\top \equiv \pi^{\top}_V(x)$ and $x^\perp\equiv \pi_V^{\perp}(x)$. We then compute
    \begin{align*}
        \textnormal{div}_V(X) & = \zeta^\prime(|x|/\rho)\frac{|x^\top|^2}{|x|\rho}f(x) + \zeta(|x|/\rho)\<x,\nabla_V f \> + n\zeta(|x|/\rho)f(x)\\
        & = \zeta^\prime f\frac{|x|}{\rho} - \zeta^\prime f\frac{|x^\perp|^2}{|x|\rho} + (n+2)\zeta f - \zeta\<x,\nabla_V^\perp f \>,
    \end{align*}
    and therefore, by stationarity of $V$,
    \begin{align*}
        \frac{d}{d\rho}\left(\rho^{-n-2}\int\zeta(|x|/\rho)f\, d\|V\|\right) & = -\rho^{-n-3}\int (n+2)\zeta f + \zeta^\prime f\frac{|x|}{\rho}\, d\|V\|\\
        & = -\rho^{-n-3}\int \zeta^\prime f \frac{|x^\perp|^2}{|x|\rho} + \zeta\<x,\nabla^\perp_V f \>\, d\|V\|.
    \end{align*}
    Integrating this over $\rho\in [\sigma,\tau]$, and then taking $\zeta\to \mathbf{1}_{[1/2,1]}$, we get
    \begin{align*}
        &\tau^{-n-2}\int_{A_{\tau,\tau/2}}f\, d\|V\| - \sigma^{-n-2}\int_{A_{\sigma,\sigma/2}}f\, d\|V\|\\
        & \hspace{7em} = \int_{A_{\tau,\sigma}}\frac{|x^\perp|^2f}{|x|^{n+4}}\, d\|V\| - \int_{A_{\tau/2,\sigma/2}}\frac{|x^\perp|^2f}{|x|^{n+4}}\, d\|V\|\\
        & \hspace{15em} - \int^\tau_\sigma\rho^{-n-3}\int_{A_{\rho,\rho/2}}\<x,\nabla^\perp_V f\>\, d\|V\|d\rho.
    \end{align*}
    
    From the $2$-homogeneity of $f$, we have
        \[
        |\<x,\nabla^\perp_V f\>| = |\<\nabla^\perp_V(x),D f\>| \leq 4|x^\perp||x|M^2,
        \]
        and so we can bound the last term like
        \begin{align*}
            &\left|\int^\tau_\sigma\rho^{-n-3}\int_{A_{\rho,\rho/2}}\<x,\nabla^\perp_V f\>\, d\|V\| d\rho\right|\\
            & \leq 4M^2\int^\tau_{\sigma}\rho^{-n-3}\left(\int_{A_{\rho,\rho/2}}|x|^2\, d\|V\|\right)^{1/2}\left(\int_{A_{\rho,\rho/2}}|x^\perp|^2\, d\|V\|\right)^{1/2}\, d\rho\\
            & \leq 4M^2\Lambda^{1/2}\int^\tau_\sigma\rho^{-1}\left(\rho^{-n-2}\int_{A_{\rho,\rho/2}}|x^\perp|^2\, d\|V\|\right)^{1/2}\, d\rho\\
            & \leq 4M^2\Lambda^{1/2}\int^\tau_\sigma\rho^{-1}\left(\int_{A_{\rho,\rho/2}}\frac{|x^\perp|^2}{|x|^{n+2}}\, d\|V\|\right)^{1/2}\, d\rho\\
            & \leq 4M^2\Lambda^{1/2}\log(\tau/\sigma)\left(\int_{A_{\tau,\sigma/2}}\frac{|x^\perp|^2}{|x|^{n+2}}\, d\|V\|\right)^{1/2}.
        \end{align*}
        This completes the proof.
\end{proof}

\section{Monotonicity of frequency}\label{sec:freq-mono-appendix}

In this section we prove almost-monotonicity of Krummel--Wickramasekera's planar frequency function, in the form stated in Theorem \ref{thm:freq-mono}. Our argument is a refinement of that in \cite{KW23a}. The main technical differences are how the graphical scales $\sigma(\xi)$ are chosen, and how the error of $D^\prime$ is handled, which together allow us to show the ``scale invariant'' differential inequality \eqref{eqn:freq-mono-concl1} for general small excess $E$. Throughout this section we fix $P := \R^n\times\{0^m\}$. Also, for a $q$-valued function $w = \sum_{i=1}^q\llbracket w_i\rrbracket:B_1^n\to \cA_q(\R^m)$ and $p\in \R^m$, we define $w\ominus p := \sum^q_{i=1}\llbracket w_i-p\rrbracket$.

\begin{lemma}\label{lem:dir-minz-recenter}
    Given $a>0$, there is $b = b(n,m,q,a)>0$ such that the following holds. Let $w\in W^{1,2}(B^n_1;\cA_q(\R^m))$ be a Dirichlet-minimizing function. Then:
    \begin{align}
       & \|w\|_{L^2(B_{1/2})} \geq a\|w\|_{L^2(B_1)}\nonumber\\
       &\hspace{5em}\Longrightarrow\quad \|w\|_{L^2(B_{1/20}(\xi))}\geq b\|w\|_{L^2(B_1)} \quad \text{for all }\xi\in B_{1/5};\label{eqn:recenter-concl1}
    \end{align}
    and
    \begin{align}
        & \|Dw\|_{L^2(B_{1/2})} \geq a\|Dw\|_{L^2(B_1)}\nonumber\\
        &\hspace{5em}\Longrightarrow\quad \|Dw\|_{L^2(B_{1/20}(\xi))}\geq b\|Dw\|_{L^2(B_1)}\quad \text{for all }\xi\in B_{1/5}.\label{eqn:recenter-concl2}
    \end{align}
\end{lemma}

\begin{proof}
    Implication \eqref{eqn:recenter-concl1} is a straightforward contradiction argument, as in \cite[Lemma 4.7]{KW23a}. For implication \eqref{eqn:recenter-concl2}, we can again argue by contradiction as follows. If the claim were false, then for some fixed $a > 0$ we could find sequences of numbers $b_j \to 0$, points $\xi_j \to \xi \in \overline{B_{1/5}^n}$ and Dirichlet-minimizers $w_j : B_1^n \to \cA_Q(\R^m)$ so that 
    \begin{align}
    &\|Dw_j\|_{L^2(B_1)} = 1 \quad \text{and}\quad \|D w_j\|_{L^2(B_{1/2})} \geq a \label{eqn:dir-minz-1} \\
    &\quad \text{yet} \quad \|Dw_j\|_{L^2(B_{1/40}(\xi))} \leq \|D w_j\|_{L^2(B_{1/20}(\xi_i))} \leq b_j \to 0. \label{eqn:dir-minz-2}
    \end{align}
    
    We may then apply the compactness theorem for Dirichlet-minimizers with unit Dirichlet-energy (see \cite[Theorem 2.15(1)]{Almgren_2000}) to deduce that, up to passing to another subsequence, we can find $J\in\{1,\dotsc,Q\}$, positive integers $k_1,\dotsc,k_J\in\{1,\dotsc,Q\}$ with $k_1+\cdots+k_J = Q$, points $p_{j}^{(1)},\dotsc,p_j^{(J)}\in \R^m$, and Dirichlet-minimizers $g^{(1)},\dotsc,g^{(J)} :B^n_1\to \cA_{k_i}(\R^m)$ such that we have
    \[
    w_j = \sum^J_{i=1}w_j^{(i)} \ominus p_j^{(i)}
    \]
    where $w_j^{(i)}$ are Dirichlet-minimizers such that for every $i = 1, \ldots, J$, we have $w_j^{(i)} \to g^{(i)}$ locally uniformly in $B_1^n$ and $|Dw_j^{(i)}| \to |Dg^{(i)}|$ in $L^2_{\text{loc}}(B_1)$.  In particular,
    \[
    \sum^J_{i=1}\int_{B^n_\sigma(y)}|Dg^{(i)}|^2 = \lim_{j\to\infty}\int_{B^n_{\sigma}(y)}|Dw_j|^2
    \]
    for each ball $B^n_\sigma(y)\subset\subset B^n_1$. Our assumptions \eqref{eqn:dir-minz-1}, \eqref{eqn:dir-minz-2} then give
    \[
    \sum^J_{i=1}\int_{B_{1/2}}|Dg^{(i)}|^2 \geq a\quad \text{yet}\quad \sum^J_{i=1}\int_{B_{1/40}(\xi)}|Dg^{(i)}|^2=0.
    \]
    However, unique continuation (cf.~\cite[Theorem 2.15(2)]{Almgren_2000}) then implies that we must have $\|Dg^{(i)}\|_{L^2(B_1)}=0$ for all $i$, giving the desired contradiction.
\end{proof}

\begin{lemma}\label{lem:H-helper}
    Let $q\in \N$, $a>0$, and let $T$ be an area-minimizing integral $n$-current in $C_8(0)$ satisfying
    \[
    (\del T)\res C_8(0) = 0, \quad \pi_{P\#}T = q\llbracket B^n_8(0)\rrbracket,\quad E_{T,P}(0,8)\leq\eps^2,
    \]
    \[
    \int_{C_{1/2}(0)}d_P^2\, d\|T\| \geq a^2\int_{C_8(0)}d_P^2\, d\|T\|.
    \]
    Then, provided $\eps = \eps(n,m,q,a)\in (0,1)$ is sufficiently small, there is $b = b(n,m,q,a)>0$ for which
    \[
    \int_{C_{1/20}(\xi)}d_P^2\, d\|T\| \geq b^2\int_{C_8(0)}d_P^2\, d\|T\| \quad \text{for all }\xi\in B^n_{1/5}(0).
    \]
\end{lemma}
\begin{proof}
    By Lemma \ref{lem:mass-by-L2} and Lemma \ref{lem:Linfty-by-L2} we have
    \[
    \sup_{x\in \spt\|T\|\cap C_7(0)} d_P(x)\leq c\eps \quad \text{and}\quad \cE_{T,P}(0,6)\leq c\eps^2.
    \]
    From these, ensuring $\eps = \eps(n,m,q)$ is sufficiently small, we can apply Almgren's Lipschitz approximation (Theorem \ref{thm:Lipschitz-approx}) to approximate $T\res C_1$ by a $q$-valued Lipschitz graph. The proof then proceeds verbatim as in \cite[Lemma 4.8]{KW23a}, with the help of Lemma \ref{lem:dir-minz-recenter}.
\end{proof}

\begin{lemma}\label{lem:H-error}
    Let $T$ be an area-minimizing integral $n$-current in $C_{2\rho}(0)$. Suppose that
    \begin{equation}\label{eqn:H-error-hyp1}
        (\del T)\res C_{2\rho}(0) = 0, \quad \Theta_T(0,\rho/2)\geq 1/2, \quad \Theta_T(0,2\rho) \leq \theta,
    \end{equation}
    \begin{equation*}
        E_{T,P}(0,2\rho) \leq E \leq \delta^2.
    \end{equation*}
    Then, provided $\delta = \delta(n,m,\theta)\in (0,1)$ is sufficiently small we have
    \begin{equation}\label{eqn:H-error-concl1}
        -\rho^{1-n}\int d_P^2\|\pi_T-\pi_P\|^2\frac{1}{r}\phi^\prime(r,\rho)\, d\|T\| \leq cE^\gamma H_{T,P}(\rho),
    \end{equation}
    and therefore
    \begin{equation}\label{eqn:H-error-concl2}
        \left|H^\prime_{T,P}(\rho)+2\rho^{-n}\int\<x,\pi_{P^\perp}(\nabla^Tr)\>\phi^\prime(r/\rho)\, d\|T\|\right| \leq cE^\gamma \rho^{-1}H_{T,P}(\rho),
    \end{equation}
    and
    \begin{equation}\label{eqn:H-error-concl3}
        H_{T,P}(\rho) \leq 2\rho^{1-n}\int_{C_\rho\setminus C_{\rho/2}}d_P^2\frac{1}{r}\, d\|T\| \leq (1+cE^{\gamma})H_{T,P}(\rho).
    \end{equation}
    Here, $c = c(n,m,\theta)>0$ and $\gamma = \gamma(n,m,\theta)>0$.
\end{lemma}
\begin{proof}
    We follow \cite[Lemma 4.9]{KW23a}. From Lemma \ref{lem:mass-by-L2} and Lemma \ref{lem:Linfty-by-L2} we have
    \begin{equation}\label{eqn:H-error-0.6}
        \cE_{T,P}(0,\rho)\leq c(n,m)E \quad \text{and} \quad \sup_{x\in \spt\|T\|\cap C_\rho(0)}d_P(x) \leq c(n,m)E^{1/2}\rho.
    \end{equation}
    From Lemma \ref{lem:zero-mass}, provided $\delta = \delta(n,m,\theta)$ is sufficiently small, we have
    \begin{equation}\label{eqn:H-error-0.5}
        \pi_{P\#}T = q\llbracket B^n_{2\rho}(0)\rrbracket
    \end{equation}
    for some integer $|q|\leq\theta$. If $q=0$ then by \eqref{eqn:H-error-0.6} and Lemma \ref{lem:zero-mass} we must have $T\res C_\rho(0)=0$, but this contradicts our lower density assumption from \eqref{eqn:H-error-hyp1}. So, after possibly flipping the orientation of $T$ as necessary, we may assume $1\leq q\leq\theta$. Note that \eqref{eqn:H-error-0.5} implies that
    \begin{equation}\label{eqn:H-error-1}
        \pi_{P\#}(T\res C_\sigma(\xi)) = q\llbracket B^n_\sigma(\pi_P(\xi))\rrbracket
    \end{equation}
    for any $C_\sigma(\xi)\subset C_{2\rho}(0)$.

    Now fix any $\eps\in (0,1/2)$. For $\xi\in A^n_{\rho,\rho/2}$, define
    \[
    \bar{\sigma}(\xi):= \sup\left(\{0\}\cup\left\{\sigma^{-n-2}\int_{C_{\sigma}(\xi)}d_P^2\, d\|T\| \geq \chi^2(\sigma/\rho)^{2\eps} E\right\}\right),
    \]
    for $\chi^2:= 64^{n+8}$. If $\sigma \in [\rho/32,\rho/2]$, then we can estimate
    \begin{align*}
        \int_{C_\sigma(\xi)}d_P^2\, d\|T\| \leq \int_{C_{2\rho}(0)}d_P^2\, d\|T\| \leq E(2\rho)^{n+2} & \leq 64^{n+4}(\sigma/\rho)^{2\eps}E \sigma^{n+2}\\
        & < \chi^2(\sigma/\rho)^{2\eps}E\sigma^{n+2}
    \end{align*}
    where the last inequality is strict by our choice of $\chi$. This shows us that we must have $\bar{\sigma}(\xi)\leq\rho/32$ for all $\xi\in A^n_{\rho,\rho/2}$.

    For $\xi \in A^n_{\rho,\rho/2}$ and definition of $\bar{\sigma}(\xi)$ we have, by Lemma \ref{lem:Linfty-by-L2},
    \[
    \sup_{x\in \spt\|T\|\cap C_{16\sigma}(\xi)}d_P(x) \leq c(n)\chi(\sigma/\rho)^\eps E^{1/2}\cdot\sigma \quad \text{for all }\sigma\in [\bar{\sigma}(\xi),\rho/32],
    \]
    which in particular implies, if setting $\Xi:= \{\xi\in A^n_{\rho,\rho/2}:\bar{\sigma}(\xi)>0\}$, that $\spt\|T\|\cap ((A^n_{\rho,\rho/2}\setminus\Xi)\times\R^m)\subset P$.

    Now choose a Vitali collection $\{\xi_i\}_i\subset\Xi$ so that, writing $\sigma_i:= \bar{\sigma}(\xi_i)$, the balls $\{B_{2\sigma_i/5}(\xi_i)\}_i$ are disjoint and the balls $\{B_{2\sigma_i}(\xi_i)\}_i$ cover $\Xi$. We then have from Lemma \ref{lem:mass-by-L2} that
    \[
    \cE_{T,P}(\xi_i,8\sigma_i) \leq c(n,m)E_{T,P}(\xi_i,16\sigma_i) \leq 16 \chi^2(\sigma_i/\rho)^{2\eps}E,
    \]
    and therefore in view of \eqref{eqn:H-error-1}, provided $\delta = \delta(n,m,\theta)$ is sufficiently small, we can apply Almgren's Lipschitz approximation (Theorem \ref{thm:Lipschitz-approx}) to find Lipschitz functions $u_i:B^n_{2\sigma_i}(\xi_i) \to \cA_q(\R^m)$ and subsets $K_i\subset B^n_{2\sigma_i}(\xi_i)$ satisfying
    \begin{equation}\label{eqn:H-error-lip1}
        \Lip(u_i) \leq c(\sigma_i/\rho)^{\gamma\eps} E^\gamma;
    \end{equation}
    \begin{equation}\label{eqn:H-error-lip2}
        T\res (K_i\times \R^m) = \mathbf{v}(u_i) \res (K_i\times\R^m);
    \end{equation}
    \begin{equation}\label{eqn:H-error-lip3}
        |B^n_{2\sigma_i}(\xi_i)\setminus K_i| + \|T\|((B^n_{2\sigma_i}(\xi_i)\setminus K_i)\times\R^m) \leq c(\sigma_i/\rho)^{2\eps(1+\gamma)}E^{(1+\gamma)}\sigma_i^n.
    \end{equation}
    Here, $c = c(n,m,\theta)$ and $\gamma = \gamma(n,m,\theta)>0$. We can then estimate
    \begin{align*}
        \int_{C_{2\sigma_i}(\xi_i)}d_P^2&\|\pi_T-\pi_P\|^2\, d\|T\|\\
        &\leq c\int_{K_i\cap B^n_{2\sigma_i}(\xi_i)}|u_i|^2|Du_i|^2\, dx + c\sigma_i^{n+2}(\sigma_i/\rho)^{2\eps(2+\gamma)}E^{2+\gamma}\\
        & \leq c(\sigma_i/\rho)^{2\gamma\eps}E^{\gamma}\int_{K_i\cap B^n_{2\sigma_i}(\xi_i)}|u_i|^2\, dx + c\sigma_i^{n+2}(\sigma_i/\rho)^{2\eps(2+\gamma)}E^{2+\gamma}\\
        & \leq c(\sigma_i/\rho)^{2\gamma\eps}E^{\gamma}\int_{C_{2\sigma_i}(\xi_i)}d_P^2\, d\|T\| + c\sigma_i^{n+2}(\sigma_i/\rho)^{2\eps(2+\gamma)}E^{2+\gamma}\\
        & \leq c(n,m,\theta)\sigma_i^{n+2}(\sigma_i/\rho)^{2\eps(1+\gamma)}E^{1+\gamma}.
    \end{align*}
    Noting that by definition of $\bar{\sigma}$ we have
    \[
    E_{T,P}(\xi_i,\sigma_i) \geq \chi^2(\sigma_i/\rho)^{2\eps}E \geq \frac{1}{16}E_{T,P}(\xi_i,16\sigma_i),
    \]
    and noting that since $\xi_i\in C_\rho\setminus C_{\rho/2}$ we can find $B^n_{\sigma_i/10}(y_i)\subset B^n_{2\sigma_i/5}(\xi_i)\cap A^n_{\rho,\rho/2}$ and then apply Lemma \ref{lem:H-helper} at scale $C_{2\sigma_i}(\xi_i)$ to deduce
    \begin{align*}
        -\rho\int_{C_{2\sigma_i/5}(\xi_i)}d_P^2\frac{1}{r}\phi^\prime(r/\rho)\, d\|T\| & \geq \int_{C_{\sigma_i/10}(y_i)}d_P^2\, d\|T\|\\
        & \geq \frac{1}{c}\int_{C_{16\sigma_i}(\xi_i)}d_P^2\, d\|T\|\\
        & \geq\frac{1}{c}\int_{C_{\sigma_i}(\xi_i)}d_P^2\, d\|T\|\\
        & \geq \frac{1}{c}\sigma_i^{n+2}(\sigma_i/\rho)^{2\eps}E
    \end{align*}
    for $c = c(n,m,\theta)$. We then compute
    \begin{align*}
        -\rho\int d_P^2\|\pi_T-\pi_P\|^2&\frac{1}{r}\phi^\prime(r/\rho)\, d\|T\|\\
        & \leq 4\sum_i\int_{C_{2\sigma_i}(\xi_i)}d_P^2\|\pi_T-\pi_P\|^2\, d\|T\|\\
        & \leq c\sum_i\sigma_i^{n+2}(\sigma_i/\rho)^{2\eps(1+\gamma)}E^{1+\gamma}\\
        & \leq c\sum_i\left(-\rho\int_{C_{2\sigma_i/5}(\xi_i)}d_P^2\frac{1}{r}\phi^\prime(r/\rho)\, d\|T\|\right)\cdot (\sigma_i/\rho)^{2\eps\gamma}E^{\gamma}\\
        & \leq c(n,m,\theta)E^{\gamma}\left(-\rho\int d_P^2\frac{1}{r}\phi^\prime(r/\rho)\, d\|T\|\right).
    \end{align*}
    Combining the above estimate with $|\nabla r|^2\geq 1-\|\pi_T-\pi_P\|^2$, and choosing $\delta = \delta(n,m,\theta)$ sufficiently small, we get
    \begin{equation}\label{eqn:H-error-2}
    -\rho^{1-n}\int d_P^2\frac{1}{r}\phi^\prime(r/\rho)\, d\|T\| \leq (1+cE^{\gamma})H_{T,P}(\rho) \leq 2H_{T,P}(\rho),
    \end{equation}
    which again when combined with the previous estimate implies \eqref{eqn:H-error-concl1}. The proof of \eqref{eqn:H-error-concl2} then follows as in \cite[Lemma 4.9]{KW23a}. Finally, the first inequality in \eqref{eqn:H-error-concl3} is immediate from the definition of $H_{T,P}(\rho)$, whilst the second follows from \eqref{eqn:H-error-2}.
\end{proof}

\begin{lemma}\label{lem:D-helper}
    Let $q\in \N$, $a>0$, and $T$ be an area-minimizing integral $n$-current in $C_8(0)$ satisfying
    \begin{equation}\label{eqn:D-helper-hyp1}
        (\del T)\res C_8(0) = 0, \quad \pi_{P\#}T = q\llbracket B^n_8(0)\rrbracket,\quad \cE_{T,P}(0,8)\leq \eps^2,
    \end{equation}
    \[
    \int_{C_{1/2}(0)}|\vec{T}-\vec{P}|^2\, d\|T\| \geq a^2\int_{C_8(0)}|\vec{T}-\vec{P}|^2\, d\|T\|.
    \]
    Then, provided $\eps = \eps(n,m,q,a)\in (0,1)$ is sufficiently small, there is $b = b(n,m,q,a)>0$ so that
    \[
    \int_{C_{1/20}(\xi)}\|\pi_T-\pi_P\|^2\, d\|T\| \geq b^2\int_{C_8(0)}|\vec{T}-\vec{P}|^2\, d\|T\| \quad \text{for all }\xi\in B^n_{1/5}(0).
    \]
\end{lemma}

\begin{proof}
    The proof is the same as in \cite[Lemma 4.10]{KW23a}, using Lemma \ref{lem:dir-minz-recenter} and Almgren's Lipschitz approximation (Theorem \ref{thm:Lipschitz-approx}). Note that one does not require $T$ to be close to $q\llbracket B^n_1(0)\rrbracket$ in $L^2$ height to apply Theorem \ref{thm:Lipschitz-approx}, it suffices to assume \eqref{eqn:D-helper-hyp1}.
\end{proof}

\begin{lemma}\label{lem:D-error}
Let $T$ be an area-minimizing integral $n$-current in $C_{2\rho}(0)$ with:
\[
(\del T)\res C_{2\rho}=0, \quad \Theta_T(0,\rho/2)\geq 1/2,\quad \Theta_T(0,2\rho)\leq\theta,
\]
\[
E_{T,P}(0,2\rho)\leq E \leq \delta^2.
\]
Then, provided $\delta = \delta(n,m,\theta)$ is sufficiently small, we can find positive constants $c^\prime<c^{\prime\prime}<c$ and $\gamma_1<\gamma$ depending only on $n,m,\theta$ such that
\begin{align}
    &\rho^{1-n}\int_{C_\rho(0)}|\vec{T}-\vec{P}|^4\, d\|T\| \label{eqn:D-error-concl1}\\
    & \leq c^\prime E^{\gamma-\gamma_1} D_{T,P}(\rho)\frac{d}{d\rho}\left[(\rho^{-2}D_{T,P}(\rho))^{\gamma_1}\right] + cE^{\gamma}\rho^{-1}D_{T,P}(\rho),\nonumber
\end{align}
and, subsequently,
\begin{align}
    &\left|D^\prime_{T,P}(\rho) + 2\rho^{-n}\int\frac{|\pi_{P^\perp}(\nabla^T r)|^2}{|\nabla^T r|^2} r\phi^\prime(r/\rho)\, d\|T\|\right| \label{eqn:D-error-concl2} \\
    & \hspace{1em} \leq c^{\prime\prime}E^{\gamma-\gamma_1}D_{T,P}(\rho)\frac{d}{d\rho}\left[(\rho^{-2}D_{T,P}(\rho))^{\gamma_1}\right] + cE^{\gamma}\rho^{-1}D_{T,P}(\rho).\nonumber
\end{align}
\end{lemma}

\begin{proof}
    As in the proof of Lemma \ref{lem:H-error}, we can assume
    \begin{equation}\label{eqn:D-error-1}
        \pi_{P\#}T = q\llbracket B^n_{2\rho}(0)\rrbracket
    \end{equation}
    for some integer $1\leq q\leq\theta$. From Lemma \ref{lem:mass-by-L2} we have
    \[
    \cE_{T,P}(0,\rho) \equiv \frac{1}{2}\rho^{-n}\int_{C_\rho(0)}|\vec{T}-\vec{P}|^2\, d\|T\| \leq c(n,m)E.
    \]
    Fix any $\eps\in (0,1/2)$ and set $\chi^2:= 64^{n+8}$. For each $\xi\in B^n_\rho$, define
    \[
    \bar{\sigma}(\xi):= \sup\left(\{0\}\cup \left\{\sigma\in (0,\rho/2]:\sigma^{-n}\int_{C_{\sigma}(\xi)}|\vec{T}-\vec{P}|^2\, d\|T\| \geq \chi^2(\sigma/\rho)^{2\eps}E\right\}\right).
    \]
    Just as in Lemma \ref{lem:H-error}, our choice of $\chi$ gives $\bar{\sigma}(\xi)\leq\rho/32$ for all $\xi\in B_\rho^n$.

    Set $\Xi:= \{\xi\in B^n_\rho:\bar{\sigma}(\xi)>0\}$. Noting that for any $\xi\in B^n_\rho$ we have, by definition of $\bar{\sigma}(\xi)$,
    \begin{equation}\label{eqn:D-error-2}
        \cE_{T,P}(\xi,\sigma)\leq c(n)(\sigma/\rho)^{2\eps}E \quad \text{for all }\sigma\in [\bar{\sigma}(\xi),\rho/2],
    \end{equation}
    we see that in particular $\vec{T}(x) = \vec{P}$ for $\|T\|$-a.e.~$x\in (B^n_\rho\setminus \Xi)\times\R^m$.

    Choose a Vitali collection $\{\xi_i\}_i\subset\Xi$ so that, writing $\sigma_i := \bar{\sigma}_i(\xi_i)$, the balls $\{B_{2\sigma_i}(\xi_i)\}_i$ form a cover of $\Xi$ and the balls $\{B_{2\sigma_i/5}(\xi_i)\}_i$ are disjoint. From \eqref{eqn:D-error-2} and \eqref{eqn:D-error-1}, we can apply Almgren's Lipschitz approximation (Theorem \ref{thm:Lipschitz-approx}) in each cylinder $C_{2\sigma_i}(\xi_i)$ to obtain Lipschitz functions $u_i:B^n_{2\sigma_i}(\xi_i)\to \cA_q(\R^m)$ and subsets $K_i\subset B^n_{2\sigma_i}(\xi_i)$ satisfying the same estimates as in \eqref{eqn:H-error-lip1}, \eqref{eqn:H-error-lip2}, \eqref{eqn:H-error-lip3}.

    We now estimate
    \begin{align*}
        \int_{C_{2\sigma_i}(\xi_i)}&|\vec{T}-\vec{P}|^4\, d\|T\|\\
        &\leq c\int_{K_i\cap B^n_{2\sigma_i}(\xi_i)}|Du_i|^4\, dx + c\sigma_i^n(\sigma_i/\rho)^{2\eps(1+\gamma)}E^{1+\gamma}\\
        & \leq c(\sigma_i/\rho)^{2\eps\gamma}E^{\gamma}\int_{K_i\cap B^n_{2\sigma_i}(\xi_i)}|Du_i|^2\, dx + c\sigma_i^n(\sigma_i/\rho)^{2\eps(1+\gamma)}E^{1+\gamma}\\
        & \leq c\sigma_i^n(\sigma_i/\rho)^{2\eps(1+\gamma)}E^{1+\gamma}.
    \end{align*}
    Since $\xi_i\in B^n_\rho$, we can choose $B^n_{\sigma_i/5}(y_i)\subset B^n_{2\sigma_i/5}(\xi_i)\cap B_\rho^n$, and then on $B^n_{\sigma_i/10}(y_i)$ we have $\phi(r/\rho)\geq \sigma_i/(20\rho)$. Using Lemma \ref{lem:D-helper} in a similar manner to the proof of Lemma \ref{lem:H-error}, we deduce
    \begin{align*}
        (\sigma_i/(20\rho))^{-1}\int_{C_\rho(0)}\|\pi_T-\pi_P\|^2\phi(r/\rho)\, d\|T\| & \geq \int_{C_{\sigma_i/10}(y_i)}\|\pi_T-\pi_P\|^2\, d\|T\|\\
        & \geq \frac{1}{c}\int_{C_{16\sigma_i}(\xi_i)}|\vec{T}-\vec{P}|^2\, d\|T\|\\
        & \geq\frac{1}{c}(\sigma_i/\rho)^{2\eps}E\sigma_i^n
    \end{align*}
    for suitable $c = c(n,m,\theta)$. In particular, for each $i$ we have
    \[
    \rho^{-2}D_{T,P}(\rho)\geq \frac{1}{c}(\sigma_i/\rho)^{n+1+2\eps}E.
    \]
    We then estimate
    \begin{align}
        \int_{C_\rho(0)}&|\vec{T}-\vec{P}|^4\, d\|T\|\nonumber\\
        & \leq \sum_i \int_{C_{2\sigma_i}(\xi_i)}|\vec{T}-\vec{P}|^4\, d\|T\|\nonumber\\
        & \leq c\sum_i\sigma_i^n(\sigma_i/\rho)^{2\eps(1+\gamma)}E^{1+\gamma}\nonumber\\
        & \leq c\sum_i (\sigma_i/\rho)^{2\eps\gamma}E^{\gamma}\int_{C_{2\sigma_i/5}(\xi_i)\cap C_\rho(0)}\|\pi_T-\pi_P\|^2\, d\|T\|\nonumber\\
        & =  c\sum_i\big((\sigma_i/\rho)^{n+1+2\eps}E\big)^{\gamma_1}E^{\gamma-\gamma_1}\int_{C_{2\sigma_i/5}(\xi_i)\cap C_\rho(0)}\|\pi_T-\pi_P\|^2\, d\|T\|\nonumber\\
        & \leq c(\rho^{-2}D_{T,P}(\rho))^{\gamma_1}E^{\gamma-\gamma_1}\int_{C_\rho(0)}\|\pi_T-\pi_P\|^2\, d\|T\|,\label{eqn:D-error-estimate-1}
    \end{align}
    where $\gamma_1:= \frac{2\eps\gamma}{n+1+2\eps} < \gamma$ and $c = c(n,m,\theta)>0$.

    Now we have (by definition of $\phi$)
    \begin{align*}
        \frac{1}{2}\int_{C_\rho(0)}\|\pi_T-\pi_P\|^2\, d\|T\| & \leq \frac{1}{2}\int_{C_\rho(0)}\|\pi_T-\pi_P\|^2\big(\phi(r/\rho)-(r/\rho)\phi^\prime(r/\rho)\big)\, d\|T\|\\
        & = (n-1)\rho^{n-2}D_{T,P}(\rho) + \rho^{n-1}D_{T,P}^\prime(\rho),
    \end{align*}
    and also note that by Lemma \ref{lem:tilt-by-L2} (and Lemma \ref{lem:Linfty-by-L2} to pass from balls to cylinders) we have
    \[
    \rho^{-2}D_{T,P}(\rho) \leq cE,
    \]
    and so combining everything (writing $D := D_{T,P}(\rho)$ as a shorthand), we get
    \begin{align*}
        \rho^{1-n}\int_{C_\rho(0)}|\vec{T}-\vec{P}|^4\, d\|T\| & \leq c_1 E^{\gamma-\gamma_1}(\rho^{-2}D)^{\gamma_1}\big((n-1)\rho^{-1}D+D^\prime\big)\\
        & = c_1E^{\gamma-\gamma_1}\rho^{-2\gamma_1}D^{\gamma_1}\big(-2\rho^{-1}D+D^\prime\big)\\
        & \hspace{3em} + c_1 E^{\gamma-\gamma_1}(\rho^{-2}D)^{\gamma_1}(n+1)\rho^{-1}D\\
        & \leq \frac{c_1}{\gamma_1}E^{\gamma-\gamma_1}\rho^{-2\gamma_1}D^{\gamma_1}\big(-2\gamma_1 \rho^{-1}D + \gamma_1 D^\prime\big)\\
        & \hspace{3em} + cE^{\gamma}\rho^{-1}D\\
        & = c^\prime E^{\gamma-\gamma_1} D\frac{d}{d\rho}\big[(\rho^{-2}D)^{\gamma_1}\big] + cE^{\gamma}\rho^{-1}D
    \end{align*}
    This proves \eqref{eqn:D-error-concl1}. The proof of \eqref{eqn:D-error-concl2} follows from \eqref{eqn:D-error-concl1} in the same fashion as in \cite[Lemma 4.11]{KW23a}.
\end{proof}

\begin{proof}[Proof of Theorem \ref{thm:freq-mono}]
    Combine Lemma \ref{lem:H-error} and Lemma \ref{lem:D-error} as in \cite[Theorem 4.4]{KW23a} to get:
    \begin{align*}
        \frac{d}{d\rho}\log N_{T,P}(\rho) & = \frac{H_{T,P}(\rho)D^\prime_{T,P}(\rho) - H^\prime_{T,P}(\rho)D_{T,P}(\rho)}{H_{T,P}(\rho)D_{T,P}(\rho)}\\
        & \geq -c\rho^{-1}E^{\gamma}-c^\prime E^{\gamma-\gamma_1}\frac{d}{d\rho}\big[(\rho^{-2}D_{T,P}(\rho))^{\gamma_1}\big]. \qedhere
    \end{align*}
\end{proof}

Finally, we prove Corollary \ref{cor:H-control}.

\begin{proof}[Proof of Corollary \ref{cor:H-control}]
    Conclusion \eqref{eqn:H-control-concl1} follows from \eqref{eqn:freq-mono-concl2}. Indeed, we get
    \[
    (2N_1-c\eps^{2\gamma})\rho^{-1}\leq \frac{d}{d\rho}\log H_{T,P}(\rho) \leq (c\eps^{2\gamma} + 2N_2)\rho^{-1}
    \]
    and so integrating gives \eqref{eqn:H-control-concl1}. As for \eqref{eqn:H-control-concl2}, we follow \cite[Corollary 5.3]{KW23a}. Let $k$ be the largest integer $\leq \log_2(\rho_2/\rho_1)$. We then compute (using \eqref{eqn:H-error-concl3})
    \begin{align*}
        \int_{C_{\rho_2}}d_P^2\, d\|T\| & \leq \int_{C_{\rho_1}}d_P^2\, d\|T\| + \sum_{i=0}^k\int_{C_{2^{-i}\rho_2}\setminus\overline{C_{2^{-i-1}\rho_2}}}d_P^2\, d\|T\|\\
        & \leq \int_{C_{\rho_1}}d_P^2\, d\|T\| + \sum^k_{i=0}(2^{-i}\rho_2)^nH_{T,P}(2^{-i}\rho_2)\\
        & \leq \int_{C_{\rho_1}}d_P^2\, d\|T\| + \sum_{i=0}^k (2^{-i}\rho_2/\rho_1)^{2N_2+c\eps^{2\gamma}}(2^{-i}\rho_2)^nH_{T,P}(\rho_1)\\
        & \leq \left(1+4\sum^\infty_{i=0}(2^{-i})^{2N_2+c\eps^{2\gamma}+n}\right)(\rho_2/\rho_1)^{2N_2+c\eps^{2\gamma}+n}\int_{C_{\rho_1}}d_P^2\, d\|T\|\\
        & \leq \frac{8}{1-2^{-n}}(\rho_2/\rho_1)^{2N_2+c\eps^{2\gamma}+n}\int_{C_{\rho_1}}d_P^2\, d\|T\|,
    \end{align*}
    which completes the proof.
\end{proof}

\bibliographystyle{alpha}
\bibliography{bibliography}

\end{document}